\documentclass[12pt, reqno]{amsart}
\usepackage{amsmath, amsthm, amscd, amsfonts, amssymb, graphicx, xcolor, mathtools}
\usepackage[bookmarksnumbered, colorlinks, plainpages,
citecolor=teal,
linktoc=all]{hyperref}

\usepackage{tikz}
\usepackage{subcaption}

\usepackage{comment}

\usepackage{CJKutf8}

\newtheorem{theorem}{Theorem}[section]

\newtheorem{proposition}[theorem]{Proposition}
\newtheorem{corollary}[theorem]{Corollary}
\newtheorem{definition}[theorem]{Definition}
\newtheorem{notation}[theorem]{Notation}
\newtheorem{example}[theorem]{Example}

\newtheorem{problem}[theorem]{Problem}
\theoremstyle{remark}
\newtheorem{remark}[theorem]{Remark}
\numberwithin{equation}{section}

\DeclareMathOperator{\wt}{wt}
\DeclareMathOperator{\Lab}{\mathfrak{L}}
\DeclareMathOperator{\Tree}{Tree}
\DeclareMathOperator{\Fill}{Fill}
\DeclareMathOperator{\FTree}{FTree}
\DeclareMathOperator{\TreeLab}{TreeLab}
\DeclareMathOperator{\Aut}{Aut}
\DeclareMathOperator{\LabFill}{\mathfrak{F}}
\DeclareMathOperator{\Par}{Par}
\DeclareMathOperator{\Con}{Con}
\newcommand{\rmd}{\mathrm{d}}

\newcommand{\Zpos}{\mathbb{Z}_+}

\tikzset{
    blk/.style = {draw, circle, fill=black, text=white},
    wht/.style = {draw, circle, fill=white, text=black},
    mid_auto/.style = {midway, auto},
    my642/.code args= {#1,#2} {
        \begin{scope} [shift = {(#1)}, rotate = #2]
            \node[blk] (#1-42) at (1,0) {};
            \node[wht] (#1-64) at (2,0) {};
            \node[blk] (#1-6)  at (3,0) {};

            \draw (#1-6) -- (#1-64) node[mid_auto] {6};
            \draw (#1-64) -- (#1-42) node[mid_auto] {4};
            \draw (#1-42) -- (#1) node[mid_auto] {2};
        \end{scope}
    }
}

\begin{document}
\setcounter{page}{1}

%\color{black}{
%\noindent 
%{\small (Name of journal)}\hfill     {\small ISSN: ****-****}\\
%{\small Vol * (20**) ***}\hfill  {\small (url))}%}

\centerline{}

\title[Counting components of moduli space of 
HCMU spheres]
{Counting components of moduli space of HCMU spheres via weighted trees}

\author[Y. Song]{Yi Song} %$^{*}$

\address{School of Mathematical Sciences, University of Science and Technology of China, Hefei, Anhui, People's Republic of China.}
\email{\textcolor[rgb]{0.00,0.00,0.84}{sif4delta0@mail.ustc.edu.cn}}

%\address{----}
%\email{----}

%\dedicatory{This paper is dedicated to Professor ABCD}

%57M15 Relations of low-dimensional topology with graph theory [See also 05Cxx]


\date{ \today
%Received: xxxxxx; Revised: yyyyyy; Accepted: zzzzzz.
%\newline \indent $^{*}$ Corresponding author
}

\begin{abstract}
    HCMU surfaces are compact Riemann surfaces equipped with the Calabi extremal K\"{a}hler metric and a finite number of singularities.
    By using both the classical football decomposition introduced by Chen-Chen-Wu \cite{CCW05} and the description of the geometric structure of HCMU surfaces by Lu-Xu \cite{lu2025modulispacehcmusurfaces},
    we can use weighted plane trees to characterize HCMU spheres with a single integral conical angle. Moreover, 
    we obtain an explicit counting formula for the components of the moduli space of such HCMU spheres
    by enumerating some class of weighted plane trees.
    \end{abstract} 

\maketitle

\tableofcontents
\section{Introduction} \label{sec:introduction}

To find a ``best'' conformal metric on a Riemann surface, a class of K\"{a}hler metrics, called HCMU metrics, was initiated by E. Calabi in \cite{Calabi82extrm, Calabi85extrm} and X. X. Chen in \cite{Cxx98, Cxx99, chen2000obstruction}, and developed in \cite{LinZhu02,WangZhu00} etc.
In this manuscript, we will study and characterize the geometric properties of a specific class of HCMU spheres, by using weighted bi-colored plane trees.

\subsection{HCMU surfaces}

In a fixed K\"{a}hler class of a compact K\"{a}hler manifold $\mathcal{M}$, an extremal K\"{a}hler metric is introduced by E. Calabi in \cite{Calabi82extrm, Calabi85extrm}. The extremal K\"{a}hler metric is the critical point of the Calabi energy 
\[
\mathcal{C}(g)=\int_{\mathcal{M}} R^2 \rmd \rho,
\]
where $R$ is the scalar curvature of the metric $\rho$ in the K\"{a}hler class. The Euler-Lagrange equations of $E(\rho)$ is 
\begin{equation*}\label{eq:EL_eq_ext}
    R_{,\alpha\beta} = 0,\quad 1 \leq \alpha, \beta \leq \dim_{\mathbb{C}} \mathcal{M} \ .
\end{equation*}
The objective is to determine the ``best'' metric within a fixed K\"{a}hler class. When $\mathcal{M}$ is a compact Riemann surface, Calabi proved  in \cite{Calabi82extrm} that an extremal K\"{a}hler metric must have constant scalar curvature (CSC). This coincides with the classical uniformization theorem.

We must turn to surfaces with singularities to find non-CSC extremal K\"{a}hler metrics.
In \cite{Cxx99}, X.X. Chen gave a first example about non-CSC extremal K\"{a}hler metric with singularities. He also classified all extremal K\"{a}hler metrics on compact Riemann surfaces with finite cusp singularities, finite area and finite energy.

In \cite{WangZhu00}, G.F. Wang and X.H. Zhu discovered that on a surface with an extremal K\"{a}hler metric of finite energy and area, every singularity is either a weak cusp, or a conical singularity.
And now we will call an extremal K\"{a}hler metric with finite singularities on a compact Riemann surface as an HCMU (Hessian of the Curvature of the Metric is Umbilical) metric.

In \cite{chen2000obstruction}, X.X. Chen presented the obstruction theorem for non-CSC HCMU metrics with conical singularities.
The theorem tells us that a non-CSC HCMU metric without saddle points must be the HCMU football, whose curvature has two extremal points. 
Also in \cite{chen2000obstruction}, Chen gave the detailed properties of HCMU football metrics.

In \cite{lin2002explicit}, C.S. Lin and X.H. Zhu introduced a class of non-CSC HCMU metrics on $S^{2}$ with finite conical singularities of integer angles. They provided an explicit formula for such metrics with 3 parameters.
In \cite{chen2009existence}, Q. Chen and Y.Y. Wu generalized their results and obtained an explicit formula for non-CSC HCMU metrics on $S^{2}$ and $T^{2}$.
In \cite{CCW05}, Q. Chen, X.X. Chen and Y.Y. Wu proved that non-CSC HCMU surface can be divided into HCMU footballs, giving the surface a combinatorial structure. By the proof of existence of non-CSC HCMU spheres, they showed that the conditions in Chen's obstruction theorem about HCMU spheres are sufficient.

\subsection{HCMU moduli space}

Usually a moduli space is a space collecting all geometric objects of the same type. For HCMU surfaces, different genus and cone angles will vary the structure of HCMU surface. Thus we usually fix its genus and cone angles. 
We denote these cone angles as $2\pi\alpha_1, \cdots, 2\pi\alpha_n$, and the vector defined by $ \vec{\alpha} := (\alpha_1, \cdots, \alpha_n) \in \left( \mathbb{R}_{\geq 0} \setminus \{1\} \right) ^n$ is called \textbf{angle vector}.

\begin{definition}\label{defn:M_hcmu}
    The (geometric) \textbf{moduli space} $\mathcal{M}hcmu_{g,n}(\vec{\alpha})$ is the set of isometry classes of HCMU surfaces of genus $g$ with $n$ conical or cusp singularities, whose angle vector is prescribed to be $\vec{\alpha}$. Here a cusp singularity is regarded as a cone point of zero angle. 
\end{definition}

Throughout this paper,
we will concentrate on HCMU spheres with a single integral conical singularity, i,e, the angle vector is $\vec{\alpha} = (\alpha) \in \mathbb{Z}_{>1}$.
We try to describe the shape of such moduli space in the following sense. 

\begin{problem} \label{prob:geometric}
    What is the number of connected components of moduli space $\mathcal{M}hcmu_{0,1}(\alpha)$ of HCMU spheres with a single integral conical singularity?
\end{problem}

Here is some explanation on this problem. 
In \cite{MyjWzq24}, Y. Meng and Z.Q. Wei showed that such HCMU surfaces can be classified by the numbers of maximum and minimum points for the curvature. They also gave the sufficient and necessary conditions on these numbers.

\begin{theorem}[\cite{MyjWzq24}] \label{thm:exist_HCMU_sphere}
    For an HCMU sphere with a single conical singularity of angle $2\pi\alpha, \alpha > 0$, let $K$ be its curvature function.

    If the singularity is the extremal points of $K$, then the surface is an HCMU football. The singularity is the maximum point of $K$ if $\alpha > 1$, while being the minimum point if $0 \leq \alpha < 1$.

    If the singularity is the saddle point of $K$, then $2 < \alpha \in \mathbb{Z}$ and there will be $p$ maximum points of $K$ and $q$ minimum points, satisfing $\alpha = p+q-1$ and $p>q\geq1$. And one of the following conditions holds.
    \begin{enumerate} 
        \item $q=1$
        \item $q>1$ and $q \nmid p$.
    \end{enumerate}
    Such HCMU spheres exist if and only if the conditions above are satisfied.
\end{theorem}

Later in \cite{lu2025modulispacehcmusurfaces}, S.C Lu and B. Xu used a detailed version of football decomposition (see Theorem \ref{thm:hcmu_data}) to show that the numbers of maximum and minimum points are still not enough to completely classify such HCMU spheres. 
This naturally leads us to search for a more refined categorization of such HCMU spheres, and a deeper question of counting components of $\mathcal{M}hcmu_{0,1}(\alpha)$. 

In next section, we will review the combinatorial structure in \cite{lu2025modulispacehcmusurfaces}, transforming the geometric problem into combinatorial one.

\subsection{From geometric surfaces to weighted bi-colored trees} \label{subsec:from_geometric_to_WBP-tree}

Our problem will be transformed into the enumeration of a kind of trees, called weighted bi-colored plane trees.

\begin{definition} \label{def:WBP-tree}
    A \textbf{weighted bi-colored plane tree} (\textbf{WBP-tree} in short) is a triple $T = (V, E, \mathcal{W})$ satisfies following conditions.
    \begin{itemize}
        \item $(V,E)$ is an embedded tree on the Euclidean plane $\mathbb{R}^2$.
        \item The vertex set admits a partition $V = V^+\sqcup V^-$ such that each edge connect one vertex in $V^+$ and one in $V^-$.
        We regard vertices in $V^+$ as \textbf{black points}, and $V^-$ as \textbf{white points}. 
        \item $\mathcal{W} : E \to \mathbb{Z}_+$ is called the \textbf{weight function} of edges. 
    \end{itemize}
\end{definition}

\begin{definition} \label{def:WBP-tree_weight_of_vertex_and_passport}
    For a WBP-tree $T_i = (V = V^+\sqcup V^-, E, \mathcal{W})$, \textbf{the weight of a vertex} is the sum of weights of edges adjacent to it, denoted by $\wt : V \to \mathbb{Z}_+$.
\end{definition}

\begin{problem} \label{prob:combination}
    Assume $p>q \geq 1$ are positive integers satisfing $\alpha = p+q-1$. 
    Then count the number of WBP-trees satisfying the following conditions: 
    \begin{equation} \label{eq:condition_number_of_vertices}
        \#V^+ = p,\quad \#V^- = q .
    \end{equation}
    \begin{equation} \label{eq:condition_weight}
        \wt(x) = q ,\forall x \in V^+,\quad \wt(y) = p ,\forall y \in V^- ,
    \end{equation}

\end{problem}

The transformation from Problem \ref{prob:geometric} to Problem \ref{prob:combination} is achieved by the combinatorial data representation of HCMU surfaces introduced in \cite{lu2025modulispacehcmusurfaces}. 

\begin{theorem}[\cite{lu2025modulispacehcmusurfaces}]\label{thm:hcmu_data}
    There is a 1-1 correspondence between all generic genus $g\geq0$ HCMU surfaces with conical or cusp singularities and the
    following data sets on a genus $g$ smooth surface $S$: 
    \[ \left( E, V^+ \sqcup V^-; K_0, R; \mathcal{W, L} \right). \]
    where 
    \begin{enumerate}
        \item $E$ is a finite collection of arcs on $S$, dividing $S$ into topological polygons. We will regard the arcs in $E$ as edges of graph drawn on $S$. Each edges in $E$ represents a football in the decomposition. 
        \item $V^+ \sqcup V^-$ is partition of the vertices of $E$ such that every edge in $E$ connects points in different part. The points in $V^+(V^-)$ are called \textbf{black(white) points}, representing the maximum(minimum) points of the curvature of HCMU surfaces. And there will be a bi-colored graph $\mathcal{G} := (V^+ \sqcup V^-, E)$ drawn on the underlying surface.
        \item The complementary polygons of $E$, called the face of graph $\mathcal{G}$, represent saddle points. The set of faces in denoted by $\mathbb{F}$.
        \item The constant $K_0$ records the curvature at maximum points, and $R$ is a parameter essentially equivalent to the curvature at minimum points. Note that the curvatures at maximum(minimum) points are the same respectively.
        \item $\mathcal{W} : E \to \mathbb{R}_+$ records the cone angles at the maximum point in each football represented by edges of $E$. It satisfies \textbf{balance equations}. For maximum point $x \in V^+$ of cone angle $2\pi\beta$, we have 
        \begin{equation}\label{eq:balance_black}
            \sum_{e\in E(x) } \mathcal{W}(e) = \beta ,
        \end{equation}
        where $E(x)$ denotes the set of edges adjacent to vertex $x$.

        For minimum point $y \in V^-$ of cone angle $2\pi\beta$, we have 
        \begin{equation}\label{eq:balance_white}
            R\cdot \sum_{e\in E(y) } \mathcal{W}(e) = \beta .
        \end{equation} 
        \item $\mathcal{L} : \mathbb{F} \to (0, l)$ records the position of saddles points when viewed from a football, where $l$ is related to $K_0$ and $R$.
    \end{enumerate}
\end{theorem}

By Theorem \ref{thm:hcmu_data}, we will finally characterize the class of HCMU spheres using WBP-tree, which constructs a bridge from geometric surfaces to weighted bi-colored trees. 

\begin{proposition}
    There is a bijection between the set of components of HCMU moduli space $\mathcal{M}hcmu_{0,1}(\alpha)$ with $p$ maximum points and $q$ minimum points satisfing $\alpha = p+q-1$ and $p>q\geq1$ , and the set of WBP-trees satisfies conditions in Problem \ref{prob:combination}.
\end{proposition}

Let us now briefly examine the process. 

By Theorem \ref{thm:exist_HCMU_sphere}, case when the singularity is maximum point is trivial. We will only consider the case when the singularity is saddle point. The ``generic'' condition in Theorem \ref{thm:hcmu_data} is always satisfied in this case.

The singularity, or the saddle point, is represented by the only face in $\mathbb{F}$. The graph $\mathcal{G} := (V^+ \sqcup V^-, E)$ drawn on spheres, thus drawn plane, together with the weight function $\mathcal{W}$, constructs a weighted bi-colored plane graph $T := (V^+ \sqcup V^-, E, \mathcal{W})$. There are only one face of the graph $T$, thus the graph is a WBP-tree. Condition \eqref{eq:condition_weight} of the WBP-tree is a form of balance function.

In the moduli space of HCMU surfaces, the maximum curvature $K_0$ and the position function $\mathcal{L}$ can be continuously changed, while $R$ is decided by $p$ and $q$, and the WBP-tree $T := (V^+ \sqcup V^-, E, \mathcal{W})$ cannot change continuously. That is why the WBP-tree represents the component of $\mathcal{M}hcmu_{0,1}(\alpha)$.

\subsection{Main results}

By studying the combinatorial problem \ref{prob:combination}, we get an explicit formula to enumerate the components of HCMU moduli space in our basic setting.

\begin{theorem} \label{thm:main}
    Assume genus $g = 0$, angle vector $\vec{\alpha}=(\alpha) \in \mathbb{Z}_{>1}$, positive integers $p>q\geq1$ satisfing $\alpha = p+q-1$. Define
    \begin{equation} \label{eq:def_g_012}
        g_0 = \gcd(p, q) ,\quad g_1 = \gcd(p-1, q) ,\quad g_2 = \gcd(p,q-1) .
    \end{equation}

    Then the number of components of HCMU moduli space $\mathcal{M}hcmu_{0,1}(\alpha)$ with $p$ maximum points and $q$ minimum points, or the number of WBP-trees satisfing conditions in Problem \ref{prob:combination} is

    \begin{equation} \label{eq:Tree_Xi_with_pq}
        G(1) + \sum_{1<d|g_1} \varphi(d) G(d) + \sum_{1<d|g_2} \varphi(d) G(d) .
    \end{equation}

    $G(1)$ is defined by
    \begin{equation} \label{eq:def_S_1}
        S_1 = \{(n_1, \dots, n_{g_0}) \in \mathbb{Z}_{\geq 0}^{g_0} \,|\, \sum_{j=1}^{g_0}{jn_j}=g_0 \} ,
    \end{equation}
    \begin{equation} \label{eq:G(1)_sum_contribution}
        G(1) = \sum_{(n_1, \dots, n_{g_0}) \in S_1} \mathcal{C}_1(n_1, \dots, n_{g_0}) .
    \end{equation}

    And $G(d), d > 1$ is defined by
    \begin{equation} \label{eq:def_S_d}
        S_d = \{(s, n_1, \dots, n_{g_0}) \in \mathbb{Z}_{\geq 0}^{g_0 + 1} \,|\, s + \sum_{j=1}^{g_0}{jn_j} = \lfloor\frac{g_0}{d}\rfloor \} ,
    \end{equation}
    \begin{equation} \label{eq:G(d)_sum_contribution}
        G(d) = \sum_{(n_1, \dots, n_{g_0}) \in S_1} \mathcal{C}_d(s, n_1, \dots, n_{g_0}) .
    \end{equation}
    The detailed definitions of $\mathcal{C}_1(n_1, \dots, n_{g_0})$ and $\mathcal{C}_d(s, n_1, \dots, n_{g_0})$ are in \eqref{eq:def_C_1} and \eqref{eq:def_C_d} respectively, and $\varphi(n)$ is Euler's totient function.
\end{theorem}

In fact, we have actually solved a more general combinatorial problem. 
We obtain a formula to enumerate the WBP-trees with a given ``passport''.

\begin{definition}
    The \textbf{passport} of a WBP-tree is a pair of multiset $\Xi = (\Xi^+, \Xi^-)$, where $\Xi^{\pm}$ is collection of weights of $V^{\pm}$.
\end{definition}

We will denote multiset in the \textbf{power notation}. For example, the multiset $(1,1,1,1,3,3,7)$ is denoted by $1^4 3^2 7$. The number of black (white) points in the WBP-tree is $\#\Xi^+(\#\Xi^-)$. And the total number of points is $\#\Xi := \#\Xi^+ + \#\Xi^-$.

The generalized combinatorial problem is
\begin{problem} \label{prob:passport}
    What is the number of WBP-trees with a given passport $\Xi$?
\end{problem}

Then our combinatorial problem \ref{prob:combination} will be a special case of the above problem when passport $\Xi = (q^p, p^q)$ represents the conditions in Problem \ref{prob:combination}.

An algorithm for Problem \ref{prob:passport} has been given in \cite{kochetkov2015enumeration}. We generalize this results to the so-called ``labeled'' cases (see Definition \ref{def:labeled_passport} and \ref{def:labeling_of_tree_and_LWBP-tree} about labeled passports and labeled WBP-trees), and provide a formula in Section \ref{sec:divided_passport_and_formula} that makes computations easier. However, it is not an explicit formula.

\begin{theorem} \label{thm:number_of_Tree_Xi}
    Let $\Xi$ be a labeled passport. Then the number of labeled WBP-trees with labeled passport $\Xi$ is
    \begin{equation} \label{eq:number_of_Tree_Xi}
        G(1) + \sum_{i=1}^{u+v} {\sum_{1<d|g_i} \varphi(d) G(d)} ,
    \end{equation}
    where $g_i, G(d)$ will be given in Definition \ref{def:g_i_and_g_d}.
\end{theorem}

\subsection{Organization of this paper}

In Section \ref{sec:LWBP-tree}, we will give the definition of labeled passport and labeled WBP-trees. Then we will use these concepts to simplify the enumeration of WBP-trees. We will use precise language to describe the combinatorial problems we mentioned above.

In Section \ref{sec:divided_passport_and_formula}, we will conclude the formula in Throrem \ref{thm:number_of_Tree_Xi}, by studying the relationship between the symmetry of trees and its passports.

In Section \ref{sec:explicit_formula}, we will use the formula in Theorem \ref{thm:number_of_Tree_Xi} to get the explicit formula of Problem \ref{prob:combination}. Since the passport $\Xi = (q^p, p^q)$ is simple enough, we can provide explicit formulas for all parameters in Theorem \ref{thm:number_of_Tree_Xi}. Then we will get the explicit formula in Theorem \ref{thm:main}.

\subsection{Acknowledgment}

I would like to express my sincere gratitude to Bin Xu and Sicheng Lu for introducing me to this problem. I am particularly indebted to Sicheng Lu for his critical feedback during the preparation of the manuscript, which significantly improved the precision of the language.

\section{Labeled weighted bi-colored plane trees} \label{sec:LWBP-tree}

To consider the enumeration of WBP-tree, we can further consider the enumeration of WBP-tree with labels at its vertices. The benefit is that the labels can break the symmetry of tree, making the enumeration easier.

\subsection{Labeled weighted bi-colored plane tree and labeled passport} \label{subsec:def_LWBP-tree_and_labeled_passport}

We will give the definition of labeled weighted bi-colored plane tree later and the definition of labeled passport first. The labeled passport plays the same role as passport in recording the information of WBP-trees. Since there are labels at the vertices of labeled weighted bi-colored plane, there will also be some labels at the weights of labeled passport.

\begin{definition} \label{def:labeled_passport}
    A \textbf{labeled passport} is a pair of multiset $\Xi = (\Xi^+, \Xi^-)$, where $\Xi^{\pm}$ are subsets of $\Zpos \times S$ and $S$ is set of labels. Element $(K, k)$ of $\Xi^{\pm}$, where $K$ is called the weight and $k$ the label, will be denoted by $K_k$. Then $\Xi^{\pm}$ can be divided into the weight part $\Xi^{\pm}_W := \{K \,|\, K_k := (K,k) \in \Xi^{\pm}\}$ and the label part $\Xi^{\pm}_L := \{k \,|\, K_k := (K,k) \in \Xi^{\pm}\}$.
\end{definition}
    
\begin{notation} \label{note:labeled_passport}
    In the definition of labeled passport,
    \begin{enumerate}
        \item The set of labels $S$ is unimportant, because we only need the weights with different labels to be considered different. Then the re-labeling is permitted, as long as the weights with different labels still have different labels. We set that labels we will use later (including the star label $*$ and numbers in $Z_{\geq 0}$) are in $S$.
        \item We set label $0$ as default label, and will omit the label $0$ when writing a passport. Then \textbf{passport without labels is automatically a labeled passport}, with all its weights labeled by $0$. When we talk about passport with $n$ labels, we are talking about passport with $n$ non-zero labels.
        \item A labeled passport is also denoted in power notation, i.e. 
        \[
        \Xi = (\Xi^{+} = {K_{1,k_1}}^{\lambda_1} \dots {K_{u,k_u}}^{\lambda_u} , \Xi^{-} = {L_{1,l_1}}^{\mu_1} \dots {L_{v,l_v}}^{\mu_v}) ,
        \]
        where label $K_{i,k_i} = (K_i, k_i)$ and $L_{j,l_j} = (L_j, l_j)$ are pairwise different respectively.
        \item The number of black(white) points is also $\#\Xi^+(\#\Xi^-)$, and the total number of points is $\#\Xi := \#\Xi^+ + \#\Xi^-$.
    \end{enumerate}
\end{notation}

We define the labeling of WBP-tree as a bijection from its vertices to the set of weights and labels, more accurately, the labeled passport.

\begin{definition} \label{def:labeling_of_tree_and_LWBP-tree}
    \begin{enumerate}
        \item A \textbf{labeling} of a WBP-tree $T = (V^+\sqcup V^-, E, \mathcal{W})$ is a labeled passport $\Xi = (\Xi^+, \Xi^-)$ and a pair of bijections from sets to multisets 
        \begin{equation*}
            \Lab^{\pm} : V^{\pm} \to \Xi^{\pm} ,
        \end{equation*}
        Each of the labeling bijection induces two map called the weight map and label map
        \begin{equation*}
            \Lab_W^{\pm} : V^{\pm} \to \Xi^{\pm}_W \subseteq \Zpos ,\quad \Lab_L^{\pm} : V^{\pm} \to \Xi^{\pm}_L \subseteq S .
        \end{equation*}
        We require that the weight part of the labeled passport $\Xi^{\pm}_W$ is the collection of weight of $V^\pm$. We also require that the weight map satisfies the compatibility condition with the weight map of vertex 
        \begin{equation} \label{eq:condition_weight_compatibility}
            \Lab_W^{\pm}(x) = \wt(v) ,\forall x \in V^{\pm} .
        \end{equation}
        \item We call the quintuple $T_L = (V^+\sqcup V^-, E, \mathcal{W}, \Xi, \Lab^{\pm})$ a \textbf{labeled weighted bi-colored plane tree} or labeled WBP-tree (\textbf{LWBP-tree} in short). The pair $\Xi = (\Xi^+, \Xi^-)$ is called the \textbf{labeled passport} of the LWBP-tree.
    \end{enumerate}
\end{definition}

\begin{notation} \label{note:bijection_from_set_to_multiset}
    In the definition of LWBP-trees,
    \begin{enumerate}
        \item The map from set to multiset $f : A \to B = {K_1}^{\lambda_1} \dots K_u^{\lambda_u}$ can be simply defined. $f$ is injective, if $\#f^{-1}(K_i) \leq \lambda_i, \forall 1 \leq i \leq u$. $f$ is surjective, if $\#f^{-1}(K_i) \geq \lambda_i, \forall 1 \leq i \leq u$. $f$ is bijective, if it is injective and surjective.
        \item A WBP-tree $T$ have its unlabeled passport $\Xi$. Though by Notation \ref{note:labeled_passport}, $\Xi$ can be regard as a labeled passport with all weight labeled by $0$. Then there are also a natural labeling on $T$, by defining $\Lab_W^{\pm}(x) = \wt(x)$ and $\Lab_L^{\pm}(x) \equiv 0$. That is to say, \textbf{a WBP-tree is automatically a LWBP-tree}.
    \end{enumerate}
\end{notation}

The following definition give the ``the same'' relationship between LWBP-trees. It will show how we regard the trees drawn on plane, when considering the weights and labels. In other words, we are essentially talk about the equivalence class of LWBP-trees.

\begin{definition} \label{def:LWBP-tree_the_same}
    Two LWBP-tree $T_{L,i} = (V^+_i\sqcup V^-_i, E_i, \mathcal{W}_i, \Xi^\pm, \Lab^{\pm}_i), i = 1, 2$ are regarded as \textbf{the same} if there is an orientation-preserving homeomorphism $I : \mathbb{R}^2 \to \mathbb{R}^2$ such that
    \begin{enumerate}
        \item $I(V^{\pm}_1) = V^{\pm}_2$ , $I(E_1) = E_2$ as elements of graphs;
        \item $\mathcal{W}_2 \circ I = \mathcal{W}_1$ as weight functions;
        \item $\Lab^\pm_2 \circ I = \Lab^\pm_1$ as labeling bijections.
    \end{enumerate} 
\end{definition}

Why we do not require the passports are the same is that the labeling bijections have decided them. That is to say, if two LWBP-trees are the same, then their labeled passports must be the same too. Note that two WBP-trees are the same if and only if they regarded as LWBP-trees (by Notation \ref{note:bijection_from_set_to_multiset}) are also the same. We give the below notations to denote the set of tree with given passport.

\begin{definition} \label{def:Tree(Xi)}
    Assume $\Xi$ is a labeled passport. We donote the set of LWBP-tree with labeled passport $\Xi$ as $\Tree(\Xi)$.
\end{definition}

If $\Xi$ is unlabeled passport, then the set of WBP-tree is actually $\Tree(\Xi)$. We give an example to summarize above definitions.

\begin{example} \label{eg:tree_label}
    See Figure \ref{fig:tree_label}. We note the weights and labels of vertices on them.
    \begin{enumerate}
        \item Subfigure \ref{sub@subfig:tree_label_A} is a WBP-tree with passport $(1^3,3)$.
        \item Subfigure \ref{sub@subfig:tree_label_B}, \ref{sub@subfig:tree_label_C}, and \ref{sub@subfig:tree_label_D} are LWBP-trees with passport $(1_11_21_3,3)$. They are all labelings of WBP-tree in Subfigure \ref{sub@subfig:tree_label_A}.
        \item LWBP-tree in Subfigure \ref{sub@subfig:tree_label_B} is the same as Subfigure \ref{sub@subfig:tree_label_D}, because we can use rotating $2\pi/3$ as the ``the same'' equivalence.
        \item But LWBP-tree in Subfigure \ref{sub@subfig:tree_label_B} and Subfigure \ref{sub@subfig:tree_label_C} are different, because there are no orientation-preserving homeomorphism on $\mathbb{R}^2$ as the ``the same'' equivalence between them.
    \end{enumerate}
\end{example}

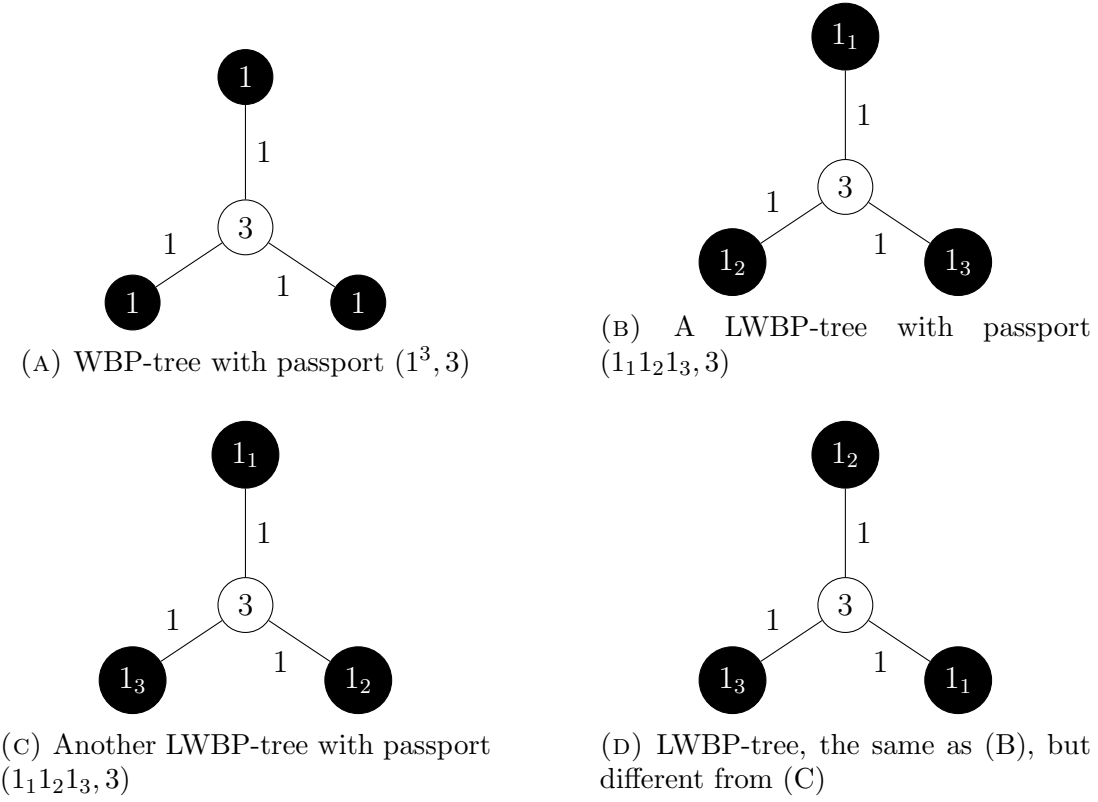
\begin{figure}[htbp]
    \centering
    \begin{subfigure}{0.45\textwidth}
        \centering
        \begin{tikzpicture}
            \node[blk] (up) at (0, 2) {1};
            \node[wht] (mid) at (0, 0) {3};
            \node[blk] (left) at (-1.5, -1) {1};
            \node[blk] (right) at (1.5, -1) {1};
            \draw (up) -- (mid) node[mid_auto] {1};
            \draw (left) -- (mid) node[mid_auto] {1};
            \draw (right) -- (mid) node[mid_auto] {1};
        \end{tikzpicture}
        \caption{WBP-tree with passport $(1^3,3)$}
        \label{subfig:tree_label_A}
    \end{subfigure}
    \hfill
    \begin{subfigure}{0.45\textwidth}
        \centering
        \begin{tikzpicture}
            \node[blk] (up) at (0, 2) {$1_1$};
            \node[wht] (mid) at (0, 0) {3};
            \node[blk] (left) at (-1.5, -1) {$1_2$};
            \node[blk] (right) at (1.5, -1) {$1_3$};
            \draw (up) -- (mid) node[mid_auto] {1};
            \draw (left) -- (mid) node[mid_auto] {1};
            \draw (right) -- (mid) node[mid_auto] {1};
        \end{tikzpicture}
        \caption{A LWBP-tree with passport $(1_11_21_3,3)$}
        \label{subfig:tree_label_B}
    \end{subfigure}

    \vspace{0.5cm}

    \begin{subfigure}{0.45\textwidth}
        \centering
        \begin{tikzpicture}
            \node[blk] (up) at (0, 2) {$1_1$};
            \node[wht] (mid) at (0, 0) {3};
            \node[blk] (left) at (-1.5, -1) {$1_3$};
            \node[blk] (right) at (1.5, -1) {$1_2$};
            \draw (up) -- (mid) node[mid_auto] {1};
            \draw (left) -- (mid) node[mid_auto] {1};
            \draw (right) -- (mid) node[mid_auto] {1};
        \end{tikzpicture}
        \caption{Another LWBP-tree with passport $(1_11_21_3,3)$}
        \label{subfig:tree_label_C}
    \end{subfigure}
    \hfill
    \begin{subfigure}{0.45\textwidth}
        \centering
        \begin{tikzpicture}
            \node[blk] (up) at (0, 2) {$1_2$};
            \node[wht] (mid) at (0, 0) {3};
            \node[blk] (left) at (-1.5, -1) {$1_3$};
            \node[blk] (right) at (1.5, -1) {$1_1$};
            \draw (up) -- (mid) node[mid_auto] {1};
            \draw (left) -- (mid) node[mid_auto] {1};
            \draw (right) -- (mid) node[mid_auto] {1};
        \end{tikzpicture}
        \caption{LWBP-tree, the same as (B), but different from (C)}
        \label{subfig:tree_label_D}
    \end{subfigure}

    \caption{WBP-tree and LWBP-trees}
    \label{fig:tree_label}
\end{figure}

\subsection{The theorem about simple labeled passports} \label{subsec:thm_YYK}

The following theorem in \cite{kochetkov2015enumeration} answers why we need LWBP-trees.

\begin{definition} \label{def:simple_partition_decomposable}
    Let $\Xi = (\Xi^{+} = {K_{1,k_1}}^{\lambda_1} \dots {K_{u,k_u}}^{\lambda_u}, \Xi^{-} = {L_{1,l_1}}^{\mu_1} \dots {L_{v,l_v}}^{\mu_v})$ be a labeled passport.
    \begin{enumerate}
        \item The passport $\Xi$ is called \textbf{simple}, if the labeled weights in $\Xi^{\pm}$ are respectively pairwise different, i.e. 
        \begin{equation*}
            \lambda_i = \mu_j = 1 , \forall 1 \leq i \leq u, 1\leq j \leq v .
        \end{equation*}
        That is to say, $\Xi^\pm$ are actually \textbf{sets}.
        \item For $n\in\Zpos$, an \textbf{n-partition} of $\Xi$ is $n$ pairs of passports $\Xi_i = (\Xi_i^+, \Xi_i^-)$ satisfing the following condition that
        \begin{equation} \label{eq:condition_partition_multiset}
            \Xi^+ = \coprod_{i=1}^{n} \Xi_i^+ ,\quad \Xi^- = \coprod_{i=1}^{n} \Xi_i^- ,
        \end{equation}
        i.e. $\Xi_i^+(\Xi_i^-)$ are partition of multiset $\Xi^+(\Xi^-)$.

        As passports, the following equations hold
        \begin{equation} \label{eq:condition_partition_sum_weight}
            \sum_{K_k \in \Xi_i^+}{K} = \sum_{L_l \in \Xi_i^-}{L},\  1 \leq i \leq n .
        \end{equation}

        We will denote the partition of passport $\Xi$ as 
        \begin{equation}
            \Xi = \coprod_{i=1}^{n} \Xi_i .
        \end{equation}
        \item Each passport $\Xi$ admits trivial 1-partition where the only partition component $\Xi_1 = \Xi$. Passport $\Xi$ is called \textbf{decomposable} if it admits a nontrivial partition.
    \end{enumerate}
\end{definition}

\begin{theorem}[\cite{kochetkov2015enumeration}]\label{thm:passport_simple}
    Let $\Xi$ be a labeled passport.
    \begin{enumerate}
        \item If $\Xi$ is simple and nondecomposable, then
        \begin{equation} \label{eq:number_Tree_simple_nondecomposable}
            \#\Tree(\Xi) = (\#\Xi - 2)! .
        \end{equation}
        \item If $\Xi$ is simple and decomposable, then $\#\Tree(\Xi)$ is a sum over all possible partitions. The contribution of each possible $n$-partition $\Xi = \coprod_{i=1}^{n} \Xi_i$ ($n\geq1$) is    
        \begin{equation} \label{eq:number_Tree_simple_decomposable}
            (-1)^{n-1} (\#\Xi-1)^{n-2} \prod_{i=1}^{n} {(\#\Xi_i-1)!} \ .
        \end{equation}
    \end{enumerate}
\end{theorem}

As we can see, we need the passport to be simple when using the above theorem. For regular passports, we cannot require them to be simple. However, we can label these passports so that they will simple.

\subsection{The filling of labeled passports} \label{subsec:filling_passport}
Theorem \ref{thm:passport_simple} can only handle the problem of simple passport. For nonsimple passport, we define the filling of it, which will be simple. And there will be relation between LWBP-trees with filled passport and with original passport.

\begin{definition} \label{def:filling_of_passport}
    The \textbf{filling} of a labeled passport $\Xi = (\Xi^+, \Xi^-)$ with $\Xi^{+} = {K_{1,k_1}}^{\lambda_1} \dots {K_{u,k_u}}^{\lambda_u} , \Xi^{-} = {L_{1,l_1}}^{\mu_1} \dots {L_{v,l_v}}^{\mu_v}$ is denoted by :
    \begin{equation}
        \Fill(\Xi) = (\Fill(\Xi)^+, \Fill(\Xi)^-)
    \end{equation}
    \begin{equation}
        \Fill(\Xi)^+ = \{K_{i,k_i,s} := (K_i, (k_i, s)) \,|\, 1 \leq i \leq u, 1 \leq s \leq \lambda_i\}
    \end{equation}
    \begin{equation}
        \Fill(\Xi)^- = \{L_{i,l_i,t} := (K_j, (k_j, t)) \,|\, 1 \leq j \leq v, 1 \leq t \leq \mu_i\}
    \end{equation}
    where labels $(k_i, s)$ are called the \textbf{double labels}. Then every labeled weights in $\Fill(\Xi)^{\pm}$ will be pairwise different, thus $\Fill(\Xi)$ is \textbf{simple} labeled passport.
\end{definition}

\begin{notation} \label{note:filling_of_passport}
    For label $(0, s)$, we will omit $0$ and denote it as $s$ for convenience. For labeled weight $K_k$ being a unique element in $\Xi$, the corresponding element $K_{k,1}$ in $\Fill(\Xi)$ will be also denoted by $K_k$ for convenience.
\end{notation}

The relation between LWBP-trees with filled passport and with original passport is as follow.

\begin{definition} \label{def:pi_map_of_FTree_to_Tree}
    Let $\Xi$ be a labeled passport.
    \begin{enumerate}
        \item We can define the bijections $\LabFill^\pm : \Fill(\Xi)^\pm \to \Xi^\pm$ by forgetting the second elements of double labels
        \begin{equation} \label{eq:def_LabFill}
            \LabFill^+ : K_{i,k_i,s} \mapsto K_{i,k_i} ,\quad \LabFill^- : L_{j,l_j,t} \mapsto L_{j,k_j} .
        \end{equation}
        \item We define the LWBP-tree with filled passport as $\FTree(\Xi) := \Tree(\Fill(\Xi))$.
        \item By using the above map $\LabFill^\pm$, we define the canonical surjection $\pi : \FTree(\Xi) \to \Tree(\Xi)$ by
        \begin{equation} \label{eq:def_pi_map}
            T_F = (V, E, \mathcal{W}, \Fill(\Xi), \Lab^{\pm}) \mapsto T_L = (V, E, \mathcal{W}, \Xi, \LabFill^\pm \circ \Lab^\pm)
        \end{equation}
    \end{enumerate}
\end{definition}

\begin{proof}
    We only need to prove that $\pi$ is well-defined. First, $\LabFill^\pm \circ \Lab^{\pm}$ are bijections, because they are composite of bijections. Second, For two LWBP-trees $T_{F,i} \in \FTree(\Xi), i=1,2$ that are the same, the ``the same'' equivalence $I$ of them is also that of $\pi(T_{F,i})$, because $\Lab^\pm_2 \circ I = \Lab^\pm_1$ implies $\LabFill^\pm \circ \Lab^\pm_2 \circ I = \LabFill^\pm \circ \Lab^\pm_1$.
\end{proof}

\begin{example}
    See Figure \ref{fig:tree_label}. Set $\Xi = (1^3,3)$, then $\Fill(\Xi) = (1_11_21_3,3)$. We have that $\pi$ map of trees in Subfigure \ref{sub@subfig:tree_label_B}, \ref{sub@subfig:tree_label_C}, and \ref{sub@subfig:tree_label_D} are all tree in Subfigure \ref{sub@subfig:tree_label_A}.
\end{example}

\subsection{The symmetry of LWBP-tree} \label{subsec:sym_of_LWBP-tree}
The surjection $\pi$ provide the relation between $\FTree(\Xi)$ and $\Tree(\Xi)$. By Theorem \ref{thm:passport_simple}, we can compute $\#\FTree(\Xi)$ because $\Fill(\Xi)$ is simple. Thus we shall make use of $\pi$ to get the quantity relationship between $\#\FTree(\Xi)$ and $\#\Tree(\Xi)$. As we will see later, this relationship is related to the symmetry of LWBP-tree.

\begin{definition} \label{def:automorphism_and_symmetry}
    Let $T_L$ be a LWBP-tree.
    \begin{enumerate}
        \item An \textbf{automorphism} of $T_L$, is a ``the same'' equivalence between $T_L$ itself. The \textbf{automorphism group} of $T_L$ is denoted by $\Aut(T_L)$.
        \item $T_L$ is \textbf{$e$-order symmetric}, if $\Aut(T_L) \cong \mathbb{Z} / e\mathbb{Z}$.
        \item For a labeled passport $\Xi$, we denote the set of LWBP-trees with passport $\Xi$ and e-order symmetry as $\Tree(\Xi, e)$.
    \end{enumerate}
\end{definition}

The following proposition is basic properties about symmetry of plane tree.

\begin{proposition} \label{prop:realize_symmetry}
    For arbitrary LWBP-tree $T_L$, there exist $e \in \Zpos$ such that $\Aut(T_L) \cong \mathbb{Z} / e\mathbb{Z}$. That is to say, every LWBP-tree is $e$-order symmetric for some $e \in \Zpos$. Moreover, every $e$-order symmetric LWBP-tree can be actually realized as an lebeled weighted embedded tree which is truly $e$-order rotational symmetric in $\mathbb{R}^2$.
\end{proposition}

Assume $\Xi$ to be a labeled passport and $\Fill(\Xi)$ its filled passport. Since the labels of tree $T_F \in \FTree(\Xi)$ are pairwise different, $\Aut(T_L)$ is trivial. But we can still talk about the ``symmetry'' of $T_F$, by using its original tree $\pi(T_F)$, while there may be some symmetry of $\pi(T_F)$ in $\Tree(\Xi)$.

\begin{definition} \label{def:symmetry_of_FTree}
    Assume $\Xi$ to be a labeled passport. We define the set of LWBP-trees with filled passport $\Fill(\Xi)$ and ``e-order symmetry'' as $\FTree(\Xi, e) := \pi^{-1}(\Tree(\Xi, e))$, though all of its elements are \textbf{not} e-order symmetry.
\end{definition}

\begin{corollary} \label{cor:Tree_eq_coprod_Tree_symmetry}
    Assume $\Xi$ to be a labeled passport. Then
        \begin{equation} \label{eq:Tree_eq_coprod_Tree_symmetry}
            \Tree(\Xi) = \coprod_{e=1}^{+\infty} \Tree(\Xi, e) ,\quad \FTree(\Xi) = \coprod_{e=1}^{+\infty} \FTree(\Xi, e) .
        \end{equation}
        Therefore
        \begin{equation} \label{eq:number_Tree_eq_sum_number_Tree_symmetry}
            \#\Tree(\Xi) = \sum_{e=1}^{+\infty} \#\Tree(\Xi, e) ,\quad \#\FTree(\Xi) = \sum_{e=1}^{+\infty} \#\FTree(\Xi, e) .
        \end{equation}

        All disjoint unions and summations above are finite.
\end{corollary}

\begin{proof}
    That every LWBP-tree is $e$-order symmetry for some $e \in \Zpos$ by Proposition \ref{prop:realize_symmetry} ,and the definition of $\FTree(\Xi, e)$ as preimages, imply \eqref{eq:Tree_eq_coprod_Tree_symmetry}.
\end{proof}

\begin{example}
    See Figure \ref{fig:tree_label}. The tree in Subfigure \ref{sub@subfig:tree_label_A} is in $\Tree(\Xi, 3)$, while trees in Subfigure \ref{sub@subfig:tree_label_B}, \ref{sub@subfig:tree_label_C},and \ref{sub@subfig:tree_label_D} are in $\FTree(\Xi, 3)$.
\end{example}

The following theorem show the quantity relationship between $\#\FTree(\Xi, e)$ and $\#\Tree(\Xi, e)$. The passport in the theorem is labeled, and the theorem about unlabeled passport in \cite{kochetkov2015enumeration} is its simple inference.

\begin{theorem} \label{thm:p_Tree_eq_e_FTree}
    Let $\Xi = (\Xi^{+} = {K_{1,k_1}}^{\lambda_1} \dots {K_{u,k_u}}^{\lambda_u}, \Xi^{-} = {L_{1,l_1}}^{\mu_1} \dots {L_{v,l_v}}^{\mu_v})$ be a labeled passport. Define
    \begin{equation} \label{eq:def_p_passport}
        p(\Xi) := \prod_{i=1}^{u} {\lambda_i}! \prod_{j=1}^{v} {\mu_j}! .
    \end{equation}
    Then the relation between $\#LTree(\Xi, e)$ and $\#\Tree(\Xi, e)$ can be expressed as 
    \begin{equation} \label{eq:p_Tree_eq_e_FTree}
        p(\Xi) \cdot \#\Tree(\Xi, e) = e \cdot \#\FTree(\Xi, e) .
    \end{equation}
\end{theorem}

\begin{proof}
    Given a LWBP-tree $T_L = (V^+\sqcup V^-, E, \mathcal{W}, \Xi^\pm, \Lab^{\pm}) \in \Tree(\Xi)$ and a pair of bijections $\Lab^\pm_F : V^\pm \to \Fill(\Xi)^\pm$, we can make use of them to construct a new LWBP-tree in $\FTree(\Xi)$, by defining $T_F := (V^+\sqcup V^-, E, \mathcal{W}, \Fill(\Xi)^\pm, \Lab^{\pm}_F)$. If the new tree also satisfies the condition $\pi(T_F) = T_L$, or equivalently $ \LabFill^\pm \circ \Lab^\pm_F = \Lab^\pm$, we call $\Lab^\pm_F$ is a \textbf{labeling filling} of $T_L$.

    We apply double counting to the set
    \begin{equation}
        \TreeLab(\Xi, e) := \{(T_L, \Lab^\pm_F) \,|\, T_L \in \Tree(\Xi, e), \Lab^\pm_F \text{ is a labeling filling of } T_L\} .
    \end{equation}
    
    \textbf{(Step 1)} We fix the tree $T_L$, and counting the number of $\Lab^\pm_F$ is a labeling filling of it. This is to find $\Lab^\pm_F$ such that $\LabFill^\pm \circ \Lab^\pm_F = \Lab^\pm$. We define $V(K_{i,k_i}) := (\Lab^+)^{-1}(K_{i,k_i})$ and $V(L_{j,l_j}) := (\Lab^-)^{-1}(L_{j,l_j})$. By definition of $\LabFill^\pm$ in \eqref{eq:def_LabFill}, $\LabFill^\pm \circ \Lab^\pm_F = \Lab^\pm$ is equivalent to that
    \[
        \Lab^+_F|_{V(K_{i,k_i})} : V(K_{i,k_i}) \to \{K_{i,k_i,s} \,|\, 1 \leq s \leq \lambda_i\} , 1 \leq i \leq u
    \]
    \[
        \Lab^-_F|_{V(L_{j,l_j})} : V(L_{j,l_j}) \to \{L_{j,l_j,t} \,|\, 1 \leq t \leq \mu_j\} , 1 \leq j \leq v
    \]
    are all bijections. It tells us the number of $\Lab^\pm_F$ is equal to the total number of bijections above. Making use of $\#V(K_{i,k_i}) = \lambda_i$ and $\#V(L_{j,l_j}) = \mu_j$, we get the total number of bijections above is exactly $p(\Xi)$. Thus 
    \begin{equation} \label{eq:TreeLab_eq_p_Tree}
        \#\TreeLab(\Xi, e) = p(\Xi) \cdot \#\Tree(\Xi, e)
    \end{equation}

    \textbf{(Step 2)} We claim that for a fixed labeling filling $\Lab^\pm_{F,1}$, there are exactly $e$ labeling filling $\Lab^\pm_{F,2}$ construct the same new trees in $\FTree(\Xi, e)$.

    \textbf{(Step 2.1)} If $T_{F,i}, i=1,2$ constructed by $\Lab^\pm_{F,i}, i=1,2$ are the same, then the ``the same'' equivalence $I$ will induce the ``the same'' equivalence of $\pi(T_{F,1}) = \pi(T_{F,2}) = T_L$ by the proof of Definition \ref{def:pi_map_of_FTree_to_Tree}, i.e. $I \in \Aut(T_L)$.
    
    \textbf{(Step 2.2)} For a fixed labeling filling $\Lab^\pm_{F,1}$, any $I \in \Aut(T_L)$ can induce a new labeling filling $\Lab^\pm_{F,2} := \Lab^\pm_{F,1} \circ I^{-1}$. It satisfies that $\Lab^\pm_{F,2} \circ I = \Lab^\pm_{F,1}$, i.e. trees $T_{F,i}$ are the same. Moreover, it satisfies that
    \[
        \LabFill^\pm \circ \Lab^\pm_{F,2} = \LabFill^\pm \circ \Lab^\pm_{F,1} \circ I^{-1} = \Lab^\pm \circ I^{-1} = \Lab^\pm ,
    \]
    i.e. it is a labeling filling of $T_L$.
    
    \textbf{(Step 2.3)} The labeling filling $\Lab^\pm_{F,2}$ constructed in (Step 2.2) is different from $\Lab^\pm_{F,1}$, unless $I = \mathbf{1}$, because $\Fill(\Xi)$ is simple passport, with all its elements pairwise different.
    
    By above steps and the definition of $\Tree(\Xi, e)$ and $T_L \in \Tree(\Xi, e)$, there are exactly $\Aut(T_L) = e$ labeling filling $\Lab^\pm_{F,2}$ construct the same new trees in $\FTree(\Xi, e)$.

    \textbf{(Step 3)} By (Step 2), we fix $T_F \in \FTree(\Xi, e)$, put $T_L = \pi(T_F)$, and there will be $e$ labeling fillings of $T_L$ construct $T_F$. It gives following equation
    \begin{equation} \label{eq:TreeLab_eq_e_FTree}
        \#\TreeLab(\Xi, e) = e \cdot \#\FTree(\Xi, e) .
    \end{equation}

    Equation \eqref{eq:TreeLab_eq_p_Tree} and \eqref{eq:TreeLab_eq_e_FTree} implies \eqref{eq:p_Tree_eq_e_FTree}.

\end{proof}

\section{Divided passports and the formula} \label{sec:divided_passport_and_formula}

Theorem \ref{thm:p_Tree_eq_e_FTree} above change the enumeration of $\Tree(\Xi)$ into that of $\FTree(\Xi, e)$, but there is still the considering of symmetry. In Theorem \ref{thm:passport_simple}, we can only compute the number of trees with simple passport, without considering symmetry. In this section, we will use method called divided passport to solve this problem, and give the formula in Theorem \ref{thm:number_of_Tree_Xi} to represent $\#Tree(\Xi)$.

\subsection{Divided passport} \label{subsec:divided_passport}
If a LWBP-tree is symmetric, by realizing it as a rotational symmetric embedded tree in $\mathbb{R}^2$ using Proposition \ref{prop:realize_symmetry}, we can divide it into some same parts. Roughly speaking, we will call the passport of one of the part the divided passport. By studying trees with divided passport, we can also know something about the original trees. 

\begin{definition} \label{def:divided_passport}
    Let $\Xi = (\Xi^{+} = {K_{1,k_1}}^{\lambda_1} \dots {K_{u,k_u}}^{\lambda_u}, \Xi^{-} = {L_{1,l_1}}^{\mu_1} \dots {L_{v,l_v}}^{\mu_v})$ be a labeled passport. Assume $d \in \mathbb{Z}_{>1}$. If for some $1 \leq i_0 \leq u$, the following conditions are satisfied
    \begin{equation} \label{eq:divided_passport_exist_condition}
        d | K_{i_0} ,\quad d | (\lambda_{i_0} - 1) ,\quad d | \lambda_i (1 \leq i \leq u \text{ and } i \ne i_0) ,\quad d | \mu_j (1 \leq j \leq v) .
    \end{equation}
    Then we define the \textbf{divided passport} $\Xi / d = ((\Xi / d)^+, (\Xi / d)^-)$ as a labeled passport
    \begin{equation} \label{eq:def_divided_passport_pos}
        (\Xi / d)^+ := (K_{i_0}/d)_* K_1^{\lambda_1/d} \dots K_{i_0-1}^{\lambda_{i_0-1}/d} K_{i_0}^{(\lambda_{i_0}-1)/d} K_{i_0+1}^{\lambda_{i_0+1}/d} \dots K_u^{\lambda_u/d} \\
    \end{equation}
    \begin{equation} \label{eq:def_divided_passport_neg}
        (\Xi / d)^- := (L_1)^{\mu_1/d} \dots (L_v)^{\mu_v/d}
    \end{equation}
    Note that $K_{i_0}/d$ is labeled by a star, and we assume there is no star label in $\Xi$ by re-labeling. It makes that $K_{i_0}/d$ is with unique label. The vertex with star label is called \textbf{symmetric center}.
    
    If for some $1 \leq j_0 \leq v$, the similar conditions hold, we can also define $\Xi / d$ using the similar approach.

    When $d = 1$, we define $\Xi / 1$ is exactly $\Xi$, which is trivial, without star label.
\end{definition}

The following proposition give some basic properties of divided passport.

\begin{proposition} \label{prop:basic_of_divided_passport}
    Assume $\Xi$ be a labeled passport and $d \in \Zpos$.
    \begin{enumerate}
        \item If $\Xi / d$ exists, $\Xi / d$ is unique.
        \item Set 
        \begin{equation} \label{eq:def_g_i}
            g_i := \\
            \begin{cases}
                \gcd(K_i; \lambda_1, \dots, \lambda_{i-1}, \lambda_i-1, \lambda_{i+1}, \dots, \lambda_u; \mu_1, \dots, \mu_v) \\
                \text{ if } 1 \leq i \leq u \\
                \gcd(L_{i-u}; \lambda_1, \dots, \lambda_u; \mu_1, \dots, \mu_{i-u-1}, \mu_{i-u}-1, \mu_{i-u+1}, \dots, \mu_v) \\
                \text{ if } u+1 \leq i \leq u+v \\
            \end{cases} .
        \end{equation}
        Then we have that $\Xi / d$ exists if and only if for some $1 \leq i_0 \leq u+v$, $d | g_{i_0}$.
        \item $(\Xi / d) / e = \Xi / de = (\Xi / e) / d$, if one of the sides of the equation exists.
    \end{enumerate}
\end{proposition}

\begin{proof}
    First, for $d = 1$, $\Xi/1 = \Xi$ by definition. For $d > 1$, it is impossible to satisfy both $d | (\lambda_{i_0} - 1)$ and $ d | \lambda_{i_0}$. Then if $\Xi / d$ exists, $\Xi / d$ is unique.

    Second one is the direct inference of conditions \eqref{eq:divided_passport_exist_condition}.

    Third, the condition of $d = 1$ or $e = 1$ is trivial. We assume $d,e > 1$, and only prove the left equation. We donote the weight with star label in $\Xi/d$ as $(K/d)_*$. Due to the uniqueness of star label in $\Xi/d$, if we need to define $(\Xi/d)/e$, then $e|(K/d)$ i.e. $de|K$ must be satisfied, and $(\Xi/d)/e$ will have labeled weight $(K/de)_*$. The rest is to use \eqref{eq:def_divided_passport_pos} and \eqref{eq:def_divided_passport_neg}, and compare these equation of $(\Xi/d)/e$ and $\Xi/de$.

\end{proof}

The following theorem will give the relationship between trees with devided passport $\Xi/d$ and with $\Xi$.

\begin{theorem} \label{thm:Tree_Xi/d/e_eq_Tree_Xi/de}
    Let $\Xi$ be a passport, and $d , e \in \mathbb{Z}_+$. We Assume all the divided passports mentioned exist.
    \begin{enumerate}
        \item There is a bijection
        \[
        i_d : \Tree(\Xi / d) \to \coprod_{d|e}{\Tree(\Xi, e)} .
        \]
        \item $i_d |_{\Tree(\Xi / d, e)} : \Tree(\Xi / d, e) \to \Tree(\Xi, de)$ is a bijection.
        \item The equation holds
        \begin{equation} \label{eq:Tree_Xi/d/e_eq_Tree_Xi/de}
            \#\Tree(\Xi, de) = \#\Tree(\Xi / d, e) .
        \end{equation}
    \end{enumerate}
\end{theorem}

\begin{proof}
    Let $\Xi = (\Xi^{+} = {K_{1,k_1}}^{\lambda_1} \dots {K_{u,k_u}}^{\lambda_u}, \Xi^{-} = {L_{1,l_1}}^{\mu_1} \dots {L_{v,l_v}}^{\mu_v})$ be a labeled passport. Assume $\Xi/d$ is with form of \eqref{eq:def_divided_passport_pos} and \eqref{eq:def_divided_passport_neg}.

    For Part (1), we need to define the map $i_d$ and its inverse map $i_d^{-1}$.

    Fix $T_d \in \Tree(\Xi/d)$, and its symmetric center labeled by star. Glue the symmetric centers of $d$ copies of $T_d$ to make a new plane tree, and the vertex glued by $d$ symmetric centers is with labeled weight $K_{i_0}$. This makes $i_d(T_d)$.

    Now fix $T \in \Tree(\Xi, e)$ with $d|e$. By Proposition \ref{prop:realize_symmetry}, we realize it as a labeled $e$-order rotational symmetric embedded tree in $\mathbb{R}^2$. Then we can divide its rotational symmetric center into $d$ parts. It makes tree $T$ divided into $d$ parts. We assume every part is the same. For one of the parts, we label the rotational symmetric center devided into $d$ parts with $(K/d)_*$. This makes $i_d^{-1}(T)$.

    It is easy to check $i_d^{-1}$ is exactly inverse map of $i_d$. Then $i_d$ is bijection.

    Part (2) is to say, for a tree with $de$-order rotational symmetry, when we divide it into $d$ same parts, one of the parts also holds $e$-order rotational symmetry.

    Part (3) is direct inference of Part (2).
\end{proof}

\begin{corollary} \label{cor:exist_Tree_Xi_d_then_exist_passport_Xi/d}
    Let $\Xi$ be a passport, and $d \in \Zpos$. If there exists a LWBP-tree in $\Tree(\Xi, d)$, then $\Xi/d$ exists.
\end{corollary}

\begin{proof}
    For $T \in \Tree(\Xi, d)$, we can devide it into $d$ same parts using the method when we define $i_d^{-1}$ in the proof of Theorem \ref{thm:Tree_Xi/d/e_eq_Tree_Xi/de}. Then one of the parts is with passport $\Xi/d$, thus $\Xi/d$ exists.
\end{proof}

\begin{example}
    See Figure \ref{fig:def_i_d}. Assume $\Xi = (3^21^2, 42^2)$ and $\Xi/2 = (3\,1, 2_*2)$. Subfigure \ref{sub@subfig:def_i_d_A} gives the example of constructing of $i_2 : \Tree(\Xi/2) \to \Tree(\Xi, 2) = \coprod_{2|e}{\Tree(\Xi, e)}$. Subfigure \ref{sub@subfig:def_i_d_B} gives another example when star label is on another vertex. If we do not put the star label on $\Xi/2$, we cannot distinguish the two ways to glue vertices, and $i_d$ cannot be defined. Moreover, without star label, we will have $\#\Tree(3\,1, 2^2) = 1$ instead of $\#\Tree(3\,1, 2_*2) = 2$, and fail to know $\#\Tree(\Xi, 2) = 2$.
\end{example}

\begin{figure}[htbp]
    \centering
    \begin{subfigure}{0.99\textwidth}
        \centering
        \begin{tikzpicture}
            \begin{scope} [shift = {(-4, -1.5)}, xscale = 0.8, yscale = 0.8]
                \node[wht] (2D) at (0, 1) {$2_*$};
            \node[blk] (1) at (1, 2) {1};
            \node[blk] (3) at (-1, 2) {3};
            \node[wht] (2U) at (-2, 3) {2};
            \draw (2D) -- (1) node[mid_auto] {1};
            \draw (2D) -- (3) node[mid_auto] {1};
            \draw (3) -- (2U) node[mid_auto] {2};
            \end{scope}
            
            \draw [->] (-2.2, 0) -- (-1.2, 0) node[mid_auto] {copy};

            \begin{scope} [shift = {(0, 0)}, xscale = 0.5, yscale = 0.5]
                \node[wht] (2D) at (0, 1) {};
                \node (2D') at (0, 0.9) {*};
                \node[blk] (1) at (1, 2) {};
                \node[blk] (3) at (-1, 2) {};
                \node[wht] (2U) at (-2, 3) {};
                \draw (2D) -- (1) node[mid_auto] {1};
                \draw (2D) -- (3) node[mid_auto] {1};
                \draw (3) -- (2U) node[mid_auto] {2};
            \end{scope}

            \draw [<->, color = red] (0, -0.2) -- (0, 0.2);

            \begin{scope} [shift = {(0, 0)}, xscale = -0.5, yscale = -0.5]
                \node[wht] (2D) at (0, 1) {};
                \node (2D') at (0, 1.1) {*};
                \node[blk] (1) at (1, 2) {};
                \node[blk] (3) at (-1, 2) {};
                \node[wht] (2U) at (-2, 3) {};
                \draw (2D) -- (1) node[mid_auto] {1};
                \draw (2D) -- (3) node[mid_auto] {1};
                \draw (3) -- (2U) node[mid_auto] {2};
            \end{scope}

            \draw [->] (1, 0) -- (2, 0) node[mid_auto] {glue};

            \node[wht] (M) at (4, 0) {4};

            \begin{scope} [shift = {(4, -0.8)}, xscale = 0.8, yscale = 0.8]
                \node[blk] (1) at (1, 2) {1};
                \node[blk] (3) at (-1, 2) {3};
                \node[wht] (2U) at (-2, 3) {2};
                \draw (M) -- (1) node[mid_auto] {1};
                \draw (M) -- (3) node[mid_auto] {1};
                \draw (3) -- (2U) node[mid_auto] {2};
            \end{scope}

            \draw [dashed, color = red] (2.4, 0) -- (5.6, 0);

            \begin{scope} [shift = {(4, 0.8)}, xscale = -0.8, yscale = -0.8]
                \node[blk] (1) at (1, 2) {1};
                \node[blk] (3) at (-1, 2) {3};
                \node[wht] (2U) at (-2, 3) {2};
                \draw (M) -- (1) node[mid_auto] {1};
                \draw (M) -- (3) node[mid_auto] {1};
                \draw (3) -- (2U) node[mid_auto] {2};
            \end{scope}

            \draw [->] (4, -1.5) -- (4, -2.5) -- (-4, -2.5) -- (-4, -1.5);
            \node (Div) at (0, -2.15) {divide};
        \end{tikzpicture}
        \caption{$\Xi = (3^21^2, 42^2)$ and $\Xi/2 = (3\,1, 2_*2)$. The subfigure gives an example of constructing of $i_2 : \Tree(\Xi/2) \to \Tree(\Xi, 2) = \coprod_{2|e}{\Tree(\Xi, e)}$.}
        \label{subfig:def_i_d_A}
    \end{subfigure}
    \vspace{0.2cm}
    \centering
    \begin{subfigure}{0.99\textwidth}
        \centering
        \begin{tikzpicture}
            \begin{scope} [shift = {(-4, -1.5)}, xscale = 0.8, yscale = 0.8]
                \node[wht] (2D) at (0, 1) {2};
            \node[blk] (1) at (1, 2) {1};
            \node[blk] (3) at (-1, 2) {3};
            \node[wht] (2U) at (-2, 3) {$2_*$};
            \draw (2D) -- (1) node[mid_auto] {1};
            \draw (2D) -- (3) node[mid_auto] {1};
            \draw (3) -- (2U) node[mid_auto] {2};
            \end{scope}

            \draw [->] (-2.2, 0) -- (-1.2, 0) node[mid_auto] {$i_2$};

            \node[wht] (M) at (2, 0) {4};

            \begin{scope} [shift = {(3.6, -2.4)}, xscale = 0.8, yscale = 0.8]
                \node[blk] (1) at (1, 2) {1};
                \node[blk] (3) at (-1, 2) {3};
                \node[wht] (2D) at (0, 1) {2};
                \draw (2D) -- (1) node[mid_auto] {1};
                \draw (2D) -- (3) node[mid_auto] {1};
                \draw (3) -- (M) node[mid_auto] {2};
            \end{scope}

            \draw [dashed, color = red] (0.4, 0) -- (3.6, 0);

            \begin{scope} [shift = {(0.4, 2.4)}, xscale = -0.8, yscale = -0.8]
                \node[blk] (1) at (1, 2) {1};
                \node[blk] (3) at (-1, 2) {3};
                \node[wht] (2D) at (0, 1) {2};
                \draw (2D) -- (1) node[mid_auto] {1};
                \draw (2D) -- (3) node[mid_auto] {1};
                \draw (3) -- (M) node[mid_auto] {2};
            \end{scope}
        \end{tikzpicture}
        \caption{another example of $i_2$ when star label is on another vertex.}
        \label{subfig:def_i_d_B}
    \end{subfigure}
    \caption{The definition of $i_d$}
    \label{fig:def_i_d}
\end{figure}
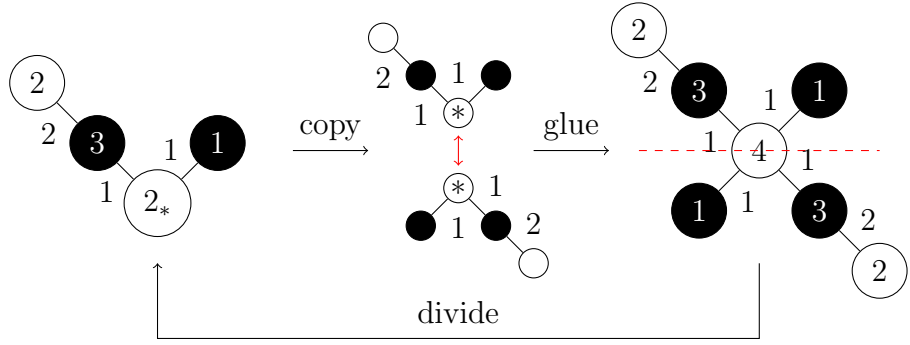
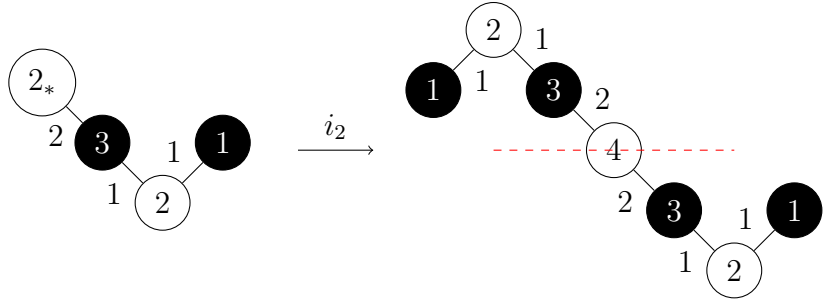

\subsection{The number of LWBP-tree with given labeled passport} \label{subsec:formula}
We now give the proof of formula in Theorem \ref{thm:number_of_Tree_Xi}. The formula only needs the value of $\#\FTree(\Xi/d)$, which can compute using Theorem \ref{thm:passport_simple}. We will restate the definition of $g_i$ in \eqref{eq:def_g_i} and state the definition of $G(d)$, which are used in Theorem \ref{thm:number_of_Tree_Xi}.

\begin{definition} \label{def:g_i_and_g_d}
    Let $\Xi = (\Xi^{+} = {K_{1,k_1}}^{\lambda_1} \dots {K_{u,k_u}}^{\lambda_u}, \Xi^{-} = {L_{1,l_1}}^{\mu_1} \dots {L_{v,l_v}}^{\mu_v})$ be a labeled passport. Assume $d \in \mathbb{Z}_{>1}$. Define
    \begin{equation*}
        g_i := \\
        \begin{cases}
            \gcd(K_i; \lambda_1, \dots, \lambda_{i-1}, \lambda_i-1, \lambda_{i+1}, \dots, \lambda_u; \mu_1, \dots, \mu_v) \\
            \text{ if } 1 \leq i \leq u \\
            \gcd(L_{i-u}; \lambda_1, \dots, \lambda_u; \mu_1, \dots, \mu_{i-u-1}, \mu_{i-u}-1, \mu_{i-u+1}, \dots, \mu_v) \\
            \text{ if } u+1 \leq i \leq u+v \\
        \end{cases} ,
    \end{equation*}
    \begin{equation} \label{eq:def_g_d}
        G(d) := \frac{\#\FTree(\Xi / d)}{d p(\Xi / d)} .
    \end{equation}
\end{definition}

\begin{proof}[Proof of Theorem \ref{thm:number_of_Tree_Xi}]
    We restate the formula in \eqref{eq:number_of_Tree_Xi} as
    \begin{equation} \label{eq:number_of_Tree_Xi_main}
        \#\Tree(\Xi) = G(1) + \sum_{i=1}^{u+v} {\sum_{1<d|g_i} \varphi(d) G(d)} ,
    \end{equation}

    By \eqref{eq:number_Tree_eq_sum_number_Tree_symmetry}, we shall compute $\#\Tree(\Xi, e)$. Define
    \begin{equation} \label{eq:def_D}
        D := \{d \in \Zpos \,|\, \exists i_0 \in \{1 , \dots, u+v\} \text{ s.t. } d|g_{i_0}\} .
    \end{equation}
    By Corollary \ref{cor:exist_Tree_Xi_d_then_exist_passport_Xi/d} and Proposition \ref{prop:basic_of_divided_passport}, $\#\Tree(\Xi, d) \ne 0$ implies $d \in D$, thus we only need to compute $\#\Tree(\Xi, d)$ when $d \in D$.

    By \eqref{eq:number_Tree_eq_sum_number_Tree_symmetry}, \eqref{eq:p_Tree_eq_e_FTree} and \eqref{eq:Tree_Xi/d/e_eq_Tree_Xi/de}, we get 
    \begin{equation} \label{eq:fg_dual_1}
        \begin{aligned}
            \#\FTree(\Xi/d) = \sum_{e=1}^{+\infty}{\#\FTree(\Xi/d, e)} = \\ \sum_{e=1}^{+\infty}{\frac{p(\Xi/d) \#\Tree(\Xi/d, e)}{e}} = \sum_{e=1}^{+\infty}{\frac{p(\Xi/d) \#\Tree(\Xi, de)}{e}} ,
        \end{aligned}
    \end{equation}
    if $\Xi/d$ exists. By Theorem \ref{thm:passport_simple}, we can compute $\#\FTree(\Xi/d)$. By Proposition \ref{prop:basic_of_divided_passport}, $\Xi/d$ exists if and only if $d \in D$. Therefore, \eqref{eq:fg_dual_1} forms a system of $\#D$ linear equations in $\#D$ variables, where $\#\Tree(\Xi, d)$ are the unknowns.

    We define
    \begin{equation} \label{eq:def_f_d}
        F(d) := \frac{\#\Tree(\Xi, d)}{d} .
    \end{equation}
    By definition of $G(d)$ in \eqref{eq:def_g_d}, we change \eqref{eq:fg_dual_1} to
    \begin{equation} \label{eq:fg_dual_2}
        \sum_{d|e} F(e) = \sum_{e=1}^{+\infty} F(de) = G(d) .
    \end{equation}
    We claim that
    \begin{equation} \label{eq:fg_dual_3}
        \sum_{1<d|g_i} (d-1) F(d) = \sum_{1<d|g_i} \varphi(d) G(d) .
    \end{equation}
    The proof is
    \[
        LHS = \sum_{1<d|g_i} {(\sum_{1<e|d} \varphi(e)) F(d)} = \sum_{1<e|g_i} {\varphi(e) \sum_{e|d|g_i} F(d)} = \\
        \sum_{1<e|g_i} \varphi(e) {G(e)} = RHS .
    \]
    Using \eqref{eq:fg_dual_2} and \eqref{eq:fg_dual_3}, we obtain \eqref{eq:number_of_Tree_Xi_main} :
    \begin{equation*}
        \begin{gathered}
            \#\Tree(\Xi) = \sum_{e=1}^{+\infty} \#\Tree(\Xi, e) = \sum_{e=1}^{+\infty} d F(d) = \\
            \sum_{e=1}^{+\infty} F(d) + \sum_{i=1}^{u+v} {\sum_{1<d|g_i} (d-1) F(d)} = G(1) + \sum_{i=1}^{u+v} {\sum_{1<d|g_i} \varphi(d) G(d)} .
        \end{gathered}
    \end{equation*}
\end{proof}

\begin{corollary} \label{cor:number_of_Tree_Xi_d}
    Let $\Xi$ be a labeled passport and $d \in \Zpos$. Then
    \begin{equation} \label{eq:number_of_Tree_Xi_d}
        \#\Tree(\Xi, d) = d \sum_{d|e} \mu(\frac{e}{d}) G(e) ,
    \end{equation}
    where $\mu(n)$ is Möbius function.
\end{corollary}

\begin{proof}
    By \eqref{eq:def_f_d}, \eqref{eq:fg_dual_2} and Möbius inversion formula,
    \[
        \#\Tree(\Xi, d) = dF(d) = d \sum_{d|e} \mu(\frac{e}{d}) G(e) .
    \]
\end{proof}

\begin{example}
    Assume $\Xi = (2^24^3, 8^2)$. Then $g_1 = g_3 = 1$ and $g_2 = 2$, i.e. $\Xi/d$ exists only when $d = 1$ or $2$, and $\Xi/2 = (2\,4\,2_*,8)$. The rest is to compute $G(1)$ and $G(2)$.

    It is easy to see that 
    \[
        p(\Xi) = 2!3!2! = 24 ,\quad p(\Xi/2) = 1!1!1!1! = 1 .
    \]

    Because $\Fill(\Xi/2) = (2_14_12_{*,1}, 8_1)$ is simple and nondecomposable, by \eqref{eq:number_Tree_simple_nondecomposable}
    \[
        \#\FTree(\Xi/2) = (4-2)! = 2 ,\quad G(2) = \frac{\#\FTree(\Xi/2)}{2p(\Xi/2)} = 1 .
    \]

    But $\Fill(\Xi) = (2_12_24_14_24_3, 8_18_2)$ is decomposable, and there will be $6$ $2$-partitions 
    \begin{equation*}
        \begin{gathered}
            (2_12_24_1, 8_1) \sqcup (4_24_3, 8_2),\quad (2_12_24_2, 8_1) \sqcup (4_14_3, 8_2),\quad (2_12_24_3, 8_1) \sqcup (4_14_2, 8_2), \\ (2_12_24_1, 8_2) \sqcup (4_24_3, 8_1),\quad (2_12_24_2, 8_2) \sqcup (4_14_3, 8_1),\quad (2_12_24_3, 8_2) \sqcup (4_14_2, 8_1) .
        \end{gathered}
    \end{equation*}

    Then

    \[
        \#\FTree(\Xi) = (7-2)! - 6(4-1)!(3-1)! = 48 ,\quad G(1) = \frac{\#\FTree(\Xi)}{p(\Xi)} = 2 .
    \]

    Then
    \[
        \#Tree(\Xi, 1) = \mu(\frac{1}{1}) G(1) + \mu(\frac{2}{1}) G(2) = 1 ,\quad \#Tree(\Xi, 3) = 2 \mu(\frac{2}{2}) G(2) = 2 .
    \]
    
    All the trees in $\Tree(\Xi)$ are in Figure \ref{fig:tree_2^24^38^2}.
\end{example}

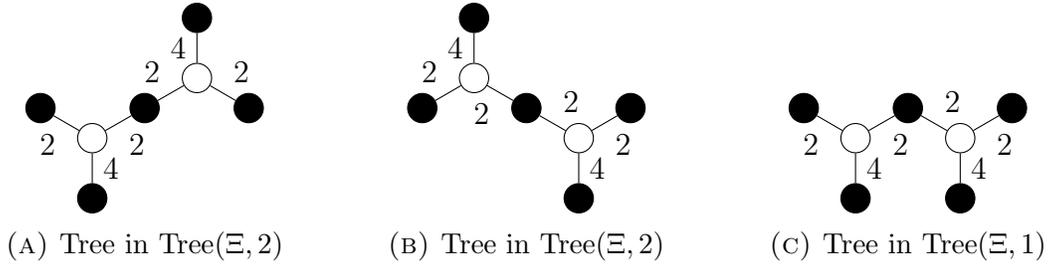
\begin{figure}[htbp]
    \centering
    \begin{subfigure}{0.3\textwidth}
        \centering
        \begin{tikzpicture}
            \node[blk] (M) at (0, 0) {};
            \begin{scope}[rotate = 0, xscale = 0.8, yscale = 0.8]
                \node[wht] (8) at (0.87, 0.5) {};
                \node[blk] (2) at (1.73, 0) {};
                \node[blk] (4) at (0.87, 1.5) {};
                \draw (M) -- (8) node[mid_auto] {2};
                \draw (8) -- (2) node[mid_auto] {2};
                \draw (8) -- (4) node[mid_auto] {4};
            \end{scope}
            \begin{scope}[rotate = 180, xscale = 0.8, yscale = 0.8]
                \node[wht] (8) at (0.87, 0.5) {};
                \node[blk] (2) at (1.73, 0) {};
                \node[blk] (4) at (0.87, 1.5) {};
                \draw (M) -- (8) node[mid_auto] {2};
                \draw (8) -- (2) node[mid_auto] {2};
                \draw (8) -- (4) node[mid_auto] {4};
            \end{scope}
        \end{tikzpicture}
        \caption{Tree in $\Tree(\Xi, 2)$}
        \label{subfig:tree_2^24^38^2_A}
    \end{subfigure}
    \hfill
    \centering
    \begin{subfigure}{0.3\textwidth}
        \centering
        \begin{tikzpicture}
            \node[blk] (M) at (0, 0) {};
            \begin{scope}[rotate = 0, xscale = -0.8, yscale = 0.8]
                \node[wht] (8) at (0.87, 0.5) {};
                \node[blk] (2) at (1.73, 0) {};
                \node[blk] (4) at (0.87, 1.5) {};
                \draw (M) -- (8) node[mid_auto] {2};
                \draw (8) -- (2) node[mid_auto, swap] {2};
                \draw (8) -- (4) node[mid_auto] {4};
            \end{scope}
            \begin{scope}[rotate = 180, xscale = -0.8, yscale = 0.8]
                \node[wht] (8) at (0.87, 0.5) {};
                \node[blk] (2) at (1.73, 0) {};
                \node[blk] (4) at (0.87, 1.5) {};
                \draw (M) -- (8) node[mid_auto] {2};
                \draw (8) -- (2) node[mid_auto, swap] {2};
                \draw (8) -- (4) node[mid_auto] {4};
            \end{scope}
        \end{tikzpicture}
        \caption{Tree in $\Tree(\Xi, 2)$}
        \label{subfig:tree_2^24^38^2_B}
    \end{subfigure}
    \hfill
    \centering
    \begin{subfigure}{0.3\textwidth}
        \centering
        \begin{tikzpicture}
            \node[blk] (M) at (0, 0) {};
            \begin{scope}[rotate = 0, xscale = -0.8, yscale = -0.8]
                \node[wht] (8) at (0.87, 0.5) {};
                \node[blk] (2) at (1.73, 0) {};
                \node[blk] (4) at (0.87, 1.5) {};
                \draw (M) -- (8) node[mid_auto] {2};
                \draw (8) -- (2) node[mid_auto] {2};
                \draw (8) -- (4) node[mid_auto] {4};
            \end{scope}
            \begin{scope}[rotate = 180, , xscale = -0.8, yscale = 0.8]
                \node[wht] (8) at (0.87, 0.5) {};
                \node[blk] (2) at (1.73, 0) {};
                \node[blk] (4) at (0.87, 1.5) {};
                \draw (M) -- (8) node[mid_auto] {2};
                \draw (8) -- (2) node[mid_auto, swap] {2};
                \draw (8) -- (4) node[mid_auto] {4};
            \end{scope}
        \end{tikzpicture}
        \caption{Tree in $\Tree(\Xi, 1)$}
        \label{subfig:tree_2^24^38^2_C}
    \end{subfigure}
    \caption{Trees in $\Tree(\Xi)$, where $\Xi = (2^24^3, 8^2)$}
    \label{fig:tree_2^24^38^2}
\end{figure}

\section{Counting the components of \texorpdfstring{$\mathcal{M}hcmu_{0,1}(\alpha)$} {\mathcal{M}hcmu{0,1}(α)}} \label{sec:explicit_formula}

We have transformed the problem of counting the components of $\mathcal{M}hcmu_{0,1}(\alpha)$ into enumeration of WBP-trees with passport
\begin{equation} \label{eq:def_passport_with_pq}
    \Xi = (p^q,q^p)
\end{equation}
by the discussion in Section \ref{subsec:from_geometric_to_WBP-tree}. We will fix the passport $\Xi$ in following discussion, and use the whole section to prove Theorem \ref{thm:main}.

\subsection{The easiest items} \label{subsec:easiest_item}

By Theorem \ref{thm:number_of_Tree_Xi}, we should only know $g_i$ and $G(d)$. It is easy to see the value of $g_i$. By definition of $G(d)$ in \eqref{eq:def_g_d}, the value of $p(\Xi/d)$ is also easy to compute, though $\#\FTree(\Xi/d)$ is not. We conclude the following proposition.
\begin{proposition} \label{prop:easiest_item}
    Let $p>q \geq 1$ are integers, passport $\Xi$ defined in \eqref{eq:def_passport_with_pq}.
    \begin{enumerate}
        \item The value of $g_i$ is the same as \eqref{eq:def_g_012}
        \begin{equation*}
            g_1 = \gcd(p-1, q) ,\quad g_2 = \gcd(p,q-1) . 
        \end{equation*}
        \item By Theorem \ref{thm:number_of_Tree_Xi}, we have the formula the same as \eqref{eq:Tree_Xi_with_pq}
        \begin{equation*}
            \#\Tree(\Xi) = G(1) + \sum_{1<d|g_1} \varphi(d) G(d) + \sum_{1<d|g_2} \varphi(d) G(d) .
        \end{equation*}
        \item The value of $p(\Xi/d)$ is as follow
        \begin{equation} \label{eq:p(Xi/d)}
            p(\Xi / d) =
            \begin{cases} 
                p!q! & \text{if } d=1 \\ 
                (\frac{p-1}{d})!(\frac{q}{d})! & \text{if } 1<d|g_1 \\ 
                (\frac{p}{d})!(\frac{p-1}{d})! & \text{if } 1<d|g_2 
            \end{cases}
        \end{equation}
    \end{enumerate}
\end{proposition}

$\#\FTree(\Xi/d)$ is hard to compute because $\Fill(\Xi/d)$ may be decomposable. Luckily, $\Fill(\Xi/d)$ are enough easy, and the partitions of them are related to $g_0 := \gcd(p,q)$, defined in \eqref{eq:def_g_012}.

We first turn to the condition when all $\Fill(\Xi/d)$ are nondecomposable, where we can use the simpler formula in \eqref{eq:number_Tree_simple_nondecomposable}. This is equivalent to $g_0 = 1$.
\begin{proposition} \label{prop:g_eq_1}
    Let $p>q \geq 1$ are integers, passport $\Xi$ defined in \eqref{eq:def_passport_with_pq}. Assume $g_0 = 1$ i.e. $\Fill(\Xi/d)$ are all nondecomposable. Then
    \begin{equation} \label{eq:FTree_Xi_d_when_g_eq_1}
        \#\FTree(\Xi/d) = (\frac{p+q-1}{d} -1)!
    \end{equation}
    \begin{equation} \label{eq:G(d)_when_g_eq_1}
        G(d) = 
        \begin{cases}
            \frac{(p+q-2)!}{p!q!} & \text{if } d=1 \\ 
            \frac{1}{p+q-1} \binom{\frac{p+q-1}{d}}{\frac{q}{d}} & \text{if } 1<d|g_1 \\ 
            \frac{1}{p+q-1} \binom{\frac{p+q-1}{d}}{\frac{p}{d}} & \text{if } 1<d|g_2 
        \end{cases}
    \end{equation}
    \begin{equation} \label{eq:Tree_Xi_when_g_eq_1}
        \#\Tree(\Xi) = \frac{(p+q-2)!}{p!q!} + \frac{1}{p+q-1} \sum_{1<d|g_1} \varphi(d) \binom{\frac{p+q-1}{d}}{\frac{q}{d}} + \sum_{1<d|g_2} \varphi(d) \binom{\frac{p+q-1}{d}}{\frac{p}{d}}
    \end{equation}
\end{proposition}

\subsection{Case when d is equal to 1} \label{subsec:d_eq_1}

When $g_0 \ne 1$, to compute $G(d)$ or $\#\FTree(\Xi/d)$, we should describe the possible partitions of $\Fill(\Xi/d)$. We first consider the condition when $d = 1$.
\begin{proposition} \label{prop:partition_Xi_satisfy}
    The partitions of multisets $\Fill(\Xi)^+ = \coprod_{i=1}^{n} \Fill(\Xi)^+_i $ and  $\Fill(\Xi)^- = \coprod_{i=1}^{n} \Fill(\Xi)^-_i$ are a partition of $\Fill(\Xi)$ if and only if the following conditions are satisfied.
    \begin{equation} \label{eq:partition_Xi_satisfy}
        \frac{q}{g_0} | {\#\Fill(\Xi)^+_i} , \frac{p}{g_0} | {\#\Fill(\Xi)^-_i} , \frac{\#\Fill(\Xi)^+_i}{q/g_0} = \frac{\#\Fill(\Xi)^-_i}{p/g_0} , \forall 1 \leq i \leq n
    \end{equation}
    The definition of $g_0 := \gcd(p,q)$ is in \eqref{eq:def_g_012}.
\end{proposition}

\begin{proof}
    By condition \eqref{eq:condition_partition_sum_weight} in Definition \ref{def:simple_partition_decomposable}, we get 
    \begin{equation*}
        \#\Fill(\Xi)^+_i \cdot q = \#\Fill(\Xi)^-_i \cdot p \text{ for all } 1 \leq i \leq n .
    \end{equation*}
    We know the minimal positive integer solution of $a p = b q$ is $a_0 = \frac{q}{g_0} , b_0 = \frac{g}{g_0}$. The above two discussions conclude the conditions \eqref{eq:partition_Xi_satisfy}.
\end{proof}

\begin{notation}
    We use the same notations of Proposition \ref{prop:partition_Xi_satisfy}. For a partition of $\Fill(\Xi)$, we denote the following multiset in power notation
    \begin{equation} \label{eq:partition_Xi_type_def}
        \{\frac{\#\Fill(\Xi)^+_i}{q/g_0} \,|\, 1 \leq i \leq n \} = 1^{n_1} \dots g_0^{n_{g_{0}}} , \text{ where } n_i \in \mathbb{Z}_{\geq 0}
    \end{equation}
    satisfies
    \begin{equation} \label{eq:partition_Xi_type_satisfy}
        \sum_{j=1}^{g_0} j n_j = g_0 , \sum_{j=1}^{g_0} n_j = n .
    \end{equation}
    And we call the vector $(n_1, \dots, n_{g_0})$ as the \textbf{type of the partition}. We denote the set of partitions of $\Fill(\Xi)$ with type $(n_1, \dots, n_{g_0})$ as $\Par_1(n_1, \dots, n_{g_0})$. And the set of possible types is defined in \eqref{eq:def_S_1}
    \begin{equation*}
        S_1 = \{(n_1, \dots, n_{g_0}) \in \mathbb{Z}_{\geq 0}^{g_0} \,|\, \sum_{j=1}^{g_0}{jn_j}=g_0 \} .
    \end{equation*}

\end{notation}

Using simple combinatorial techniques, we can get the number of partitions with given type.

\begin{proposition} \label{prop:partition_Xi_with_type_number}
    By the above notations, we have
    \begin{equation} \label{eq:partition_Xi_with_type_number}
        \# \Par_1(n_1, \dots, n_{g_0}) = \frac{1}{\prod_{j=1}^{g_0}{n_j!}}\frac{p!}{\prod_{j=1}^{g_0}{(\frac{jp}{g_0})!}^{n_j}}\frac{q!}{\prod_{j=1}^{g_0}{(\frac{jq}{g_0})!}^{n_j}}
    \end{equation}
\end{proposition}

Using above propositions, we can finally compute the value of $G(1)$ in \eqref{eq:G(1)_sum_contribution}.

\begin{proposition} \label{prop:G(1)_when_g_ne_1}
    \begin{equation*}
        G(1) = \sum_{(n_1, \dots, n_{g_0}) \in S_1} \mathcal{C}_1(n_1, \dots, n_{g_0}) .
    \end{equation*}
    \begin{equation} \label{eq:def_C_1}
        \mathcal{C}_1(n_1, \dots, n_{g_0}) = (-1)^{n-1} (p+q-1)^{n-2} \prod_{j=1}^{g_0} {\frac{1}{n_j!}[\frac{(\frac{j(p+q)}{g_0}-1)!}{(\frac{jp}{g_0})!(\frac{jq}{g_0})!}]^{n_j}} ,
    \end{equation}
    where $n = \sum_{j=1}^{g_0} n_j$.
\end{proposition}

\begin{proof}
    By Theorem \ref{thm:passport_simple}, each partition of $\Fill(\Xi)$ contributes a summation term to $\#\FTree(\Xi)$. We have known there are \# $\Par_1(n_1, \dots, n_{g_0})$ partitions with type $(n_1, \dots, n_{g_0})$, and each partition with type $(n_1, \dots, n_{g_0})$ has the same contribution. Therefore, the total contribution of partitions with type $(n_1, \dots, n_{g_0})$ is

    \begin{equation} \label{eq:def_Con_1}
        \begin{gathered}
            \Con_1(n_1, \dots, n_{g_0}) = \\
            \# \Par(n_1, \dots, n_{g_0}) (-1)^{n-1}(p+q-1)^{n-2}\prod_{j=1}^{g_0} {(j\frac{p+q}{g_0}-1)!}^{n_j} = \\
            = p!q! (-1)^{n-1} (p+q-1)^{n-2} \prod_{j=1}^{g_0} {\frac{1}{n_j!}[\frac{(\frac{j(p+q)}{g_0}-1)!}{(\frac{jp}{g_0})!(\frac{jq}{g_0})!}]^{n_j}} \\
        \end{gathered} .
    \end{equation}

    And we have
    \begin{equation} \label{eq:FTree_Xi_Con_1}
            \#\FTree(\Xi) = \sum_{(n_1, \dots, n_{g_0}) \in S_1} \Con_1(n_1, \dots, n_{g_0}) .
    \end{equation}
    By equation \eqref{eq:p(Xi/d)}, $p(\Xi) = p!q!$. Then we have
    \begin{equation*}
        G(1) = \frac{\#FTree(\Xi)}{p(\Xi)} = \sum_{(n_1, \dots, n_{g_0}) \in S_1} \frac{\Con_1(n_1, \dots, n_{g_0})}{p(\Xi)} = \sum_{(n_1, \dots, n_{g_0}) \in S_1} \mathcal{C}_1(n_1, \dots, n_{g_0})
    \end{equation*}
\end{proof}

\subsection{Case when d is greater than 1} \label{subsec:d_ne_1}
Now we consider the condition when $d > 1$. If $\Xi / d$ exists, we always assume
\begin{equation} \label{eq:assume_Xi/d}
    \Xi / d = (q^{(p-1)/d}(q/d)_*, p^{q/d}) ,
\end{equation}
because another condition is similar.

\begin{proposition} \label{prop:partition_Xi/d_satisfy}
    Assume $d \in \mathbb{Z}_{>1}$, $\Xi / d$ exists and $\Xi / d = (q^{(p-1)/d}(q/d)_*, p^{q/d})$.
    
    Then the partitions of multisets $\Fill(\Xi/d)^+ = \coprod_{i=1}^{n} \Fill(\Xi/d)^+_i $ and  $\Fill(\Xi/d)^- = \coprod_{i=1}^{n} \Fill(\Xi/d)^-_i$ are a partition of $\Fill(\Xi/d)$ if and only if
    \begin{equation} \label{eq:partition_(Xi/d)_satisfy_1}
        \frac{q}{g_0} | {\#\Fill(\Xi/d)^+_i} , \frac{p}{g_0} | {\#\Fill(\Xi/d)^-_i} , \frac{\#\Fill(\Xi/d)^+_i}{q/g_0} = \frac{\#\Fill(\Xi/d)^-_i}{p/g_0} \text{ if } 1 \leq i \leq n-1 
    \end{equation}
    and
    \begin{equation} \label{eq:partition_(Xi/d)_satisfy_2}
        \#\Fill(\Xi/d)^+_n = (s + \{\frac{g_0}{d}\}) \frac{q}{g_0} - \frac{1}{d} , \#\Fill(\Xi/d)^-_n = (s + \{\frac{g_0}{d}\})\frac{p}{g_0} \text{ for some } s \in \mathbb{Z}_{\geq 0}
    \end{equation}
    is satisfied, where the unique $(q/d)_*$ is in $\Fill(\Xi/d)^+_n$, which can always be obtained by renumbering indicator $i$.
\end{proposition}

\begin{proof}
    The condition \eqref{eq:partition_(Xi/d)_satisfy_1} can be proved by the same method of \eqref{eq:partition_Xi_satisfy}. By condition \eqref{eq:condition_partition_sum_weight} in Definition \ref{def:simple_partition_decomposable}, we can also get
    \begin{equation*}
        \#\Fill(\Xi/d)^+_n \cdot q + \frac{q}{d} = \#\Fill(\Xi/d)^-_n \cdot p .
    \end{equation*}
    We know the minimal positive integer solution of $a p + \frac{p}{d} = b q$ is $a_0 = \{\frac{g_0}{d}\} \frac{q}{g_0} - \frac{1}{d} , b_0 = \{\frac{g_0}{d}\} \frac{p}{g_0}$. As a result, 
    \begin{equation*}
        \#\Fill(\Xi/d)^+_n = a_0 + s \cdot \frac{q}{g_0} , \#\Fill(\Xi/d)^-_n = b_0 + s \cdot \frac{p}{g_0} ,
    \end{equation*}
    which is the same as condition \eqref{eq:partition_(Xi/d)_satisfy_2}
\end{proof}

\begin{notation}
    We use the same notations of Proposition \ref{prop:partition_Xi/d_satisfy}. Set
    \begin{equation} \label{eq:partition_(Xi/d)_type_def}
        \{\frac{\#\Fill(\Xi/d)^+_i}{q/g_0} | 1 \leq i \leq n-1 \} = 1^{n_1} \dots g_0^{n_{g_{0}}} ,
    \end{equation}
    satisfied
    \begin{equation} \label{eq:partition_(Xi/d)_type_satisfy}
        s + \sum_{j=1}^{g_0} j n_j =\lfloor\frac{g_0}{d}\rfloor  , 1 + \sum_{j=1}^{g_0} n_j = n .
    \end{equation}

    We call the vector $(s, n_1, \dots, n_{g_0})$ as the \textbf{type of the partition}. We denote the set of partitions of $\Fill(\Xi/d)$ with type $(s, n_1, \dots, n_{g_0})$ as $\Par_d(s, n_1, \dots, n_{g_0})$. And the set of possible types is defined in \eqref{eq:def_S_d}
    \begin{equation*}
        S_d = \{(s, n_1, \dots, n_{g_0}) \in \mathbb{Z}_{\geq 0}^{g_0 + 1} \,|\, s + \sum_{j=1}^{g_0}{jn_j} = \lfloor\frac{g_0}{d}\rfloor \} .
    \end{equation*}
\end{notation}

The following proposition is similar to Proposition \ref{prop:partition_Xi_with_type_number}. The only difference is that the passport $\Fill(\Xi/d)_n$ in partitions is special because it contains the unique $(q/d)_*$.

\begin{proposition} \label{prop:partition_(Xi/d)_with_type_number}
    By above notation, we get 
    \begin{equation} \label{eq:partition_(Xi/d)_with_type_number}
        \begin{gathered}
            \# \Par_d(s, n_1, \dots, n_{g_0}) = \\
            \frac{1}{\prod_{j=1}^{g_0}{n_j!}} \frac{((p-1)/d)!}{[(s+\{\frac{g_0}{d}\})\frac{p}{g_0}-\frac{1}{d}]!\prod_{j=1}^{g_0}{(j\frac{p}{g_0})!}^{n_j}} \frac{(q/d)!}{[(s+\{\frac{g_0}{d}\})\frac{q}{g_0}]!\prod_{j=1}^{g_0}{(j\frac{q}{g_0})!}^{n_j}}
        \end{gathered}
    \end{equation}
\end{proposition}

Using above propositions, we can finally compute the value of $G(d)$ mentioned in \eqref{eq:G(d)_sum_contribution}.

\begin{proposition} \label{prop:G(d)_when_g_ne_1}
    Suppose $\Xi / d$ exists, then
    \begin{equation*}
        G(d) = \sum_{(s, n_1, \dots, n_{g_0}) \in S_d} \mathcal{C}_d(s, n_1, \dots, n_{g_0}) .
    \end{equation*}
    \begin{equation} \label{eq:def_C_d}
        \begin{gathered}
            \mathcal{C}_d(s, n_1, \dots, n_{g_0}) = \\
            \begin{dcases} 
                (-1)^{n-1} \frac{(p+q-1)^{n-2}}{d^{n-1}} \binom{(s+\{\frac{g_0}{d}\})\frac{p+q}{g_0}-\frac{1}{d}}{(s+\{\frac{g_0}{d}\})\frac{q}{g_0}}\prod_{j=1}^{g_0}{\frac{1}{n_j!}[\frac{(\frac{j(p+q)}{g_0}-1)!}{(\frac{jp}{g_0})!(\frac{jq}{g_0})!}]^{n_j}} \\
                \hfill \text{ if } 1<d|g_1 , \\
                (-1)^{n-1} \frac{(p+q-1)^{n-2}}{d^{n-1}} \binom{(s+\{\frac{g_0}{d}\})\frac{p+q}{g_0}-\frac{1}{d}}{(s+\{\frac{g_0}{d}\})\frac{p}{g_0}}\prod_{j=1}^{g_0}{\frac{1}{n_j!}[\frac{(\frac{j(p+q)}{g_0}-1)!}{(\frac{jp}{g_0})!(\frac{jq}{g_0})!}]^{n_j}} \\
                \hfill \text{ if } 1<d|g_2 , \\
            \end{dcases}
        \end{gathered}
    \end{equation}
    where $1 + \sum_{j=1}^{g_0} n_j = n$ holds.
\end{proposition}

\begin{proof}
    Assume $1<d|g_1$ and $\Xi / d = (q^{(p-1)/d}(q/d)_*, p^{q/d})$.

    Similar to the proof of \ref{prop:G(1)_when_g_ne_1}, we can get the total contribution of partitions with type $(s, n_1, \dots, n_{g_0})$ to $\#\FTree(\Xi/d)$.
    \begin{equation} \label{eq:def_Con_d}
        \begin{gathered}
            \Con_d(s, n_1, \dots, n_{g_0}) = \\
            \# \Par(n_1, \dots, n_{g_0}) (-1)^{n-1}(\frac{p+q-1}{d})^{n-2}[(s+\{\frac{g_0}{d}\})\frac{q}{g_0}+(s+\{\frac{g_0}{d}\})\frac{p}{g_0}-\frac{1}{d}]! \\
            \prod_{j=1}^{g_0}{(\frac{j(p+q)}{g_0}-1)!}^{n_j} = \\
            ((p-1)/d)! (q/d)! (-1)^{n-1} (\frac{p+q-1}{d})^{n-2} \binom{(s+\{\frac{g_0}{d}\})\frac{p+q}{g_0}-\frac{1}{d}}{(s+\{\frac{g_0}{d}\})\frac{q}{g_0}} \\
            \prod_{j=1}^{g_0}{\frac{1}{n_j!}[\frac{(\frac{j(p+q)}{g_0}-1)!}{(\frac{jp}{g_0})!(\frac{jq}{g_0})!}]^{n_j}} \\
        \end{gathered} .
    \end{equation}

    And we have
    \begin{equation} \label{eq:FTree_Xi_Con_d}
            \#\FTree(\Xi/d) = \sum_{(s, n_1, \dots, n_{g_0}) \in S_d} \Con_d(s, n_1, \dots, n_{g_0}) .
    \end{equation}
    By equation \eqref{eq:p(Xi/d)}, $p(\Xi) = ((p-1)/d)! (q/d)!$. Then we have
    \begin{equation*}
        \begin{gathered}
            G(d) = \frac{\#FTree(\Xi/d)}{dp(\Xi/d)} = \sum_{(s, n_1, \dots, n_{g_0}) \in S_d} \frac{\Con_d(s, n_1, \dots, n_{g_0})}{dp(\Xi/d)} = \\
            \sum_{(s, n_1, \dots, n_{g_0}) \in S_d} \mathcal{C}_d(s, n_1, \dots, n_{g_0})
        \end{gathered}
    \end{equation*} 
\end{proof}

\begin{proof}[Proof of Theorem \ref{thm:main}]
    By Proposition \ref{prop:easiest_item}, we prove Equation \eqref{eq:Tree_Xi_with_pq}.
    By Proposition \ref{prop:G(1)_when_g_ne_1}, we prove Equation \eqref{eq:G(1)_sum_contribution}.
    By Proposition \ref{prop:G(d)_when_g_ne_1}, we prove Equation \eqref{eq:G(d)_sum_contribution}.
\end{proof}

\begin{remark}
    By Corollary \ref{cor:number_of_Tree_Xi_d}, we can also use $G(d)$ to compute $\#\Tree(\Xi, d)$ i.e. the number of LWBP-trees with passport $\Xi$ and $d$-order symmetry .
\end{remark}

Finally we give two examples about the formula.

\begin{example}
    We set $p = 7$ and $q = 3$, i.e. $\Xi = (3^7,7^3)$, then $g_0 = 1, g_1 = 3$ and $g_2 = 1$. Using Equation \eqref{eq:G(d)_when_g_eq_1} we get :
    \begin{equation*}
        G(1) = \frac{(7+3-2)!}{7!3!} = \frac{4}{3} ,
    \end{equation*}
    \begin{equation*}
        G(3) = \frac{1}{7+3-1} \binom{\frac{7+3-1}{3}}{\frac{3}{3}} = \frac{1}{3} ,
    \end{equation*}
    Then the total number of $Tree(\Xi)$ is 
    \begin{equation*}
        \#\Tree(\Xi) = G(1) + \varphi(3) G(3) = \frac{4}{3} + 2 \times\frac{1}{3} = 2 ,
    \end{equation*}
    where the number of $Tree(\Xi, 1)$ and $Tree(\Xi, 3)$ are
    \begin{equation*}
        \#Tree(\Xi, 1) = \mu(\frac{1}{1}) G(1) + \mu(\frac{3}{1}) G(3) = 1 ,\quad \#Tree(\Xi, 3) = 3 \mu(\frac{3}{3}) G(3) = 1 .
    \end{equation*}
    Thus there are one tree with 1-order rotational symmetry, and one with 3-order. See Figure \ref{fig:tree_p7q3}.
\end{example}

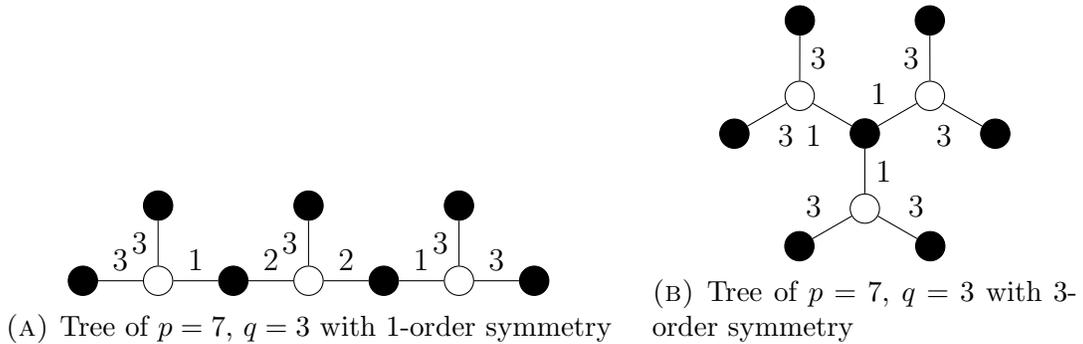
\begin{figure}[htbp]
    \centering
    \begin{subfigure}{0.59\textwidth}
        \centering

        \begin{tikzpicture}
            \node[wht] (M) at (0, 0) {};
            \node[blk] (U) at (0, 1) {};
            \node[blk] (L) at (-1, 0) {};
            \node[wht] (LL) at (-2, 0) {};
            \node[blk] (LLU) at (-2, 1) {};
            \node[blk] (LLL) at (-3, 0) {};
            \node[blk] (R) at (1, 0) {};
            \node[wht] (RR) at (2, 0) {};
            \node[blk] (RRU) at (2, 1) {};
            \node[blk] (RRR) at (3, 0) {};

            \draw (M) -- (U) node[mid_auto] {3};
            \draw (L) -- (M) node[mid_auto] {2};
            \draw (LL) -- (L) node[mid_auto] {1};
            \draw (LL) -- (LLU) node[mid_auto] {3};
            \draw (LLL) -- (LL) node[mid_auto] {3};
            \draw (M) -- (R) node[mid_auto] {2};
            \draw (R) -- (RR) node[mid_auto] {1};
            \draw (RR) -- (RRU) node[mid_auto] {3};
            \draw (RR) -- (RRR) node[mid_auto] {3};
        \end{tikzpicture}
        \caption{Tree of $p=7$, $q=3$ with 1-order symmetry}
        \label{subfig:tree_p7q3_rot1}
    \end{subfigure}
    \hfill
    \centering
    \begin{subfigure}{0.39\textwidth}
        \centering
        \begin{tikzpicture}
            \node[blk] (M) at (0, 0) {};
            \begin{scope}[rotate =0]
                \node[wht] (D) at (0, -1) {};
                \node[blk] (DL) at (-0.87, -1.5) {};
                \node[blk] (DR) at (0.87, -1.5) {};
                \draw (M) -- (D) node[mid_auto] {1};
                \draw (DL) -- (D) node[mid_auto] {3};
                \draw (D) -- (DR) node[mid_auto] {3};
            \end{scope}
            \begin{scope}[rotate =120]
                \node[wht] (D) at (0, -1) {};
                \node[blk] (DL) at (-0.87, -1.5) {};
                \node[blk] (DR) at (0.87, -1.5) {};
                \draw (M) -- (D) node[mid_auto] {1};
                \draw (DL) -- (D) node[mid_auto] {3};
                \draw (D) -- (DR) node[mid_auto] {3};
            \end{scope}
            \begin{scope}[rotate =240]
                \node[wht] (D) at (0, -1) {};
                \node[blk] (DL) at (-0.87, -1.5) {};
                \node[blk] (DR) at (0.87, -1.5) {};
                \draw (M) -- (D) node[mid_auto] {1};
                \draw (DL) -- (D) node[mid_auto] {3};
                \draw (D) -- (DR) node[mid_auto] {3};
            \end{scope}

        \end{tikzpicture}
        \caption{Tree of $p=7$, $q=3$ with 3-order symmetry}
        \label{subfig:tree_p7q3_rot3}
    \end{subfigure}
    \caption{Two trees of $p=7$, $q=3$}
    \label{fig:tree_p7q3}
\end{figure}

\begin{figure}[htbp]
    \centering
    \begin{subfigure}{0.3\textwidth}
        \centering
        % 5rot
        \begin{tikzpicture}
            \begin{scope} [xscale = 0.6, yscale = 0.6]
                \node[wht] (M) at (0,0) {};
                \tikzset{my642={M,0}}
                \tikzset{my642={M,72}}
                \tikzset{my642={M,144}}
                \tikzset{my642={M,216}}
                \tikzset{my642={M,288}}
            \end{scope}
        \end{tikzpicture}
        \caption{1 tree in $\Tree(\Xi, 5)$}
    \end{subfigure}
    \hfill
    \begin{subfigure}[t]{0.69\textwidth}
        \centering
        % 3rot1
        \begin{tikzpicture}
            \node[blk] (M) at (0,0) {};
            \begin{scope} [xscale = 0.6, yscale = 0.6]
                \node[wht] (R) at (1,0) {};
                \tikzset{my642={R,-60}}
                \node[blk] (RU) at (1.5,0.87) {};
        
                \draw (R) -- (M) node[mid_auto] {2};
                \draw (RU) -- (R) node[mid_auto] {6};
            \end{scope}
            \begin{scope} [rotate = 120, xscale = 0.6, yscale = 0.6]
                \node[wht] (R) at (1,0) {};
                \tikzset{my642={R,-60}}
                \node[blk] (RU) at (1.5,0.87) {};
        
                \draw (R) -- (M) node[mid_auto, swap] {2};
                \draw (RU) -- (R) node[mid_auto] {6};
            \end{scope}
            \begin{scope} [rotate = 240, xscale = 0.6, yscale = 0.6]
                \node[wht] (R) at (1,0) {};
                \tikzset{my642={R,-60}}
                \node[blk] (RU) at (1.5,0.87) {};
        
                \draw (R) -- (M) node[mid_auto, swap] {2};
                \draw (RU) -- (R) node[mid_auto] {6};
            \end{scope}
        \end{tikzpicture}
        \hspace{0.1cm}
        % 3rot2
        \begin{tikzpicture}
            \node[blk] (M) at (0,0) {};
            \begin{scope} [xscale = 0.6, yscale = -0.6]
                \node[wht] (R) at (1,0) {};
                \tikzset{my642={R,-60}}
                \node[blk] (RU) at (1.5,0.87) {};
        
                \draw (R) -- (M) node[mid_auto] {2};
                \draw (RU) -- (R) node[mid_auto] {6};
            \end{scope}
            \begin{scope} [rotate = 120, xscale = 0.6, yscale = -0.6]
                \node[wht] (R) at (1,0) {};
                \tikzset{my642={R,-60}}
                \node[blk] (RU) at (1.5,0.87) {};
        
                \draw (R) -- (M) node[mid_auto, swap] {2};
                \draw (RU) -- (R) node[mid_auto] {6};
            \end{scope}
            \begin{scope} [rotate = 240, xscale = 0.6, yscale = -0.6]
                \node[wht] (R) at (1,0) {};
                \tikzset{my642={R,-60}}
                \node[blk] (RU) at (1.5,0.87) {};
        
                \draw (R) -- (M) node[mid_auto, swap] {2};
                \draw (RU) -- (R) node[mid_auto] {6};
            \end{scope}
        \end{tikzpicture}
        \caption{2 trees in $\Tree(\Xi, 3)$}
    \end{subfigure}
    \vspace{0.2cm}
    \begin{subfigure}{\textwidth}
        \centering
        % quasi 3rot1
        \begin{tikzpicture}
            \node[blk] (M) at (0,0) {};
            \begin{scope} [xscale = 0.6, yscale = 0.6]
                \node[wht] (R) at (1,0) {};
                \tikzset{my642={R,-60}}
                \node[blk] (RU) at (1.5,0.87) {};

                \draw (R) -- (M) node[mid_auto] {2};
                \draw (RU) -- (R) node[mid_auto] {6};
            \end{scope}
            \begin{scope} [rotate = 120, xscale = 0.6, yscale = -0.6]
                \node[wht] (R) at (1,0) {};
                \tikzset{my642={R,-60}}
                \node[blk] (RU) at (1.5,0.87) {};

                \draw (R) -- (M) node[mid_auto, swap] {2};
                \draw (RU) -- (R) node[mid_auto] {6};
            \end{scope}
            \begin{scope} [rotate = 240, xscale = 0.6, yscale = 0.6]
                \node[wht] (R) at (1,0) {};
                \tikzset{my642={R,-60}}
                \node[blk] (RU) at (1.5,0.87) {};

                \draw (R) -- (M) node[mid_auto, swap] {2};
                \draw (RU) -- (R) node[mid_auto] {6};
            \end{scope}
        \end{tikzpicture}
        \hspace{2cm}
        % quasi 3rot2
        \begin{tikzpicture}
            \node[blk] (M) at (0,0) {};
            \begin{scope} [xscale = 0.6, yscale = -0.6]
                \node[wht] (R) at (1,0) {};
                \tikzset{my642={R,-60}}
                \node[blk] (RU) at (1.5,0.87) {};

                \draw (R) -- (M) node[mid_auto] {2};
                \draw (RU) -- (R) node[mid_auto] {6};
            \end{scope}
            \begin{scope} [rotate = 120, xscale = 0.6, yscale = -0.6]
                \node[wht] (R) at (1,0) {};
                \tikzset{my642={R,-60}}
                \node[blk] (RU) at (1.5,0.87) {};

                \draw (R) -- (M) node[mid_auto, swap] {2};
                \draw (RU) -- (R) node[mid_auto] {6};
            \end{scope}
            \begin{scope} [rotate = 240, xscale = 0.6, yscale = 0.6]
                \node[wht] (R) at (1,0) {};
                \tikzset{my642={R,-60}}
                \node[blk] (RU) at (1.5,0.87) {};

                \draw (R) -- (M) node[mid_auto, swap] {2};
                \draw (RU) -- (R) node[mid_auto] {6};
            \end{scope}
        \end{tikzpicture}
        \caption{2 trees with pseudo-3-order symmetry}
    \end{subfigure}

    \vspace{0.2cm}
    \begin{subfigure}{\textwidth}
        \centering
        % spur1
        \begin{tikzpicture}
            \begin{scope} [xscale = 0.6, yscale = 0.5]
                \node[wht] (L) at (-1,0) {};
                \tikzset{my642={L,90}}
                \tikzset{my642={L,180}}
                \tikzset{my642={L,270}}
                \node[blk] (M) at (0,0) {};
                \node[wht] (R) at (1,0) {};
                \tikzset{my642={R,-60}}
                \node[blk] (RU) at (1.5,0.87) {};

                \draw (L) -- (M) node[mid_auto, swap] {4};
                \draw (M) -- (R) node[mid_auto] {2};
                \draw (R) -- (RU) node[mid_auto] {6};
            \end{scope}
        \end{tikzpicture}
        \hspace{2cm}
        % spur2
        \begin{tikzpicture}
            \begin{scope} [xscale = 0.6, yscale = 0.5]
                \node[wht] (L) at (-1,0) {};
                \tikzset{my642={L,90}}
                \tikzset{my642={L,180}}
                \tikzset{my642={L,270}}
                \node[blk] (M) at (0,0) {};
                \node[wht] (R) at (1,0) {};
                \tikzset{my642={R,60}}
                \node[blk] (RD) at (1.5,-0.87) {};

                \draw (L) -- (M) node[mid_auto, swap] {4};
                \draw (M) -- (R) node[mid_auto] {2};
                \draw (R) -- (RD) node[mid_auto, swap] {6};
            \end{scope}
        \end{tikzpicture}
        \caption{2 spur-shaped trees}
    \end{subfigure}

    \vspace{0.2cm}
    \begin{subfigure}[t]{\textwidth}
        \centering
        % crab1
        \begin{tikzpicture}
            \node[wht] (M) at (0,0) {};
            \begin{scope} [xscale = 0.6, yscale = 0.5]
                \tikzset{my642={M,90}}
            \end{scope}

            \begin{scope} [xscale = 0.6, yscale = 0.5]
                \node[blk] (R) at (1,0) {};
                \node[wht] (RR) at (2,0) {};
                \tikzset{my642={RR,-60}}
                \node[blk] (RRU) at (2.5,0.87) {};

                \draw (R) -- (M) node[mid_auto] {4};
                \draw (RR) -- (R) node[mid_auto] {2};
                \draw (RRU) -- (RR) node[mid_auto] {6};
            \end{scope}
            \begin{scope} [xscale = -0.6, yscale = 0.5]
                \node[blk] (R) at (1,0) {};
                \node[wht] (RR) at (2,0) {};
                \tikzset{my642={RR,-60}}
                \node[blk] (RRU) at (2.5,0.87) {};

                \draw (R) -- (M) node[mid_auto] {4};
                \draw (RR) -- (R) node[mid_auto] {2};
                \draw (RRU) -- (RR) node[mid_auto] {6};
            \end{scope}
            
        \end{tikzpicture}
        \hspace{2cm}
        % crab2
        \begin{tikzpicture}
            \node[wht] (M) at (0,0) {};
            \begin{scope} [xscale = 0.6, yscale = 0.5]
                \tikzset{my642={M,90}}
            \end{scope}

            \begin{scope} [xscale = 0.6, yscale = -0.5]
                \node[blk] (R) at (1,0) {};
                \node[wht] (RR) at (2,0) {};
                \tikzset{my642={RR,-60}}
                \node[blk] (RRU) at (2.5,0.87) {};

                \draw (R) -- (M) node[mid_auto] {4};
                \draw (RR) -- (R) node[mid_auto] {2};
                \draw (RRU) -- (RR) node[mid_auto] {6};
            \end{scope}
            \begin{scope} [xscale = -0.6, yscale = 0.5]
                \node[blk] (R) at (1,0) {};
                \node[wht] (RR) at (2,0) {};
                \tikzset{my642={RR,-60}}
                \node[blk] (RRU) at (2.5,0.87) {};

                \draw (R) -- (M) node[mid_auto] {4};
                \draw (RR) -- (R) node[mid_auto] {2};
                \draw (RRU) -- (RR) node[mid_auto] {6};
            \end{scope}
            
        \end{tikzpicture}
        \vspace{0.2cm}
        \centering
        % crab3
        \begin{tikzpicture}
            \node[wht] (M) at (0,0) {};
            \begin{scope} [xscale = 0.6, yscale = 0.5]
                \tikzset{my642={M,90}}
            \end{scope}

            \begin{scope} [xscale = 0.6, yscale = 0.5]
                \node[blk] (R) at (1,0) {};
                \node[wht] (RR) at (2,0) {};
                \tikzset{my642={RR,-60}}
                \node[blk] (RRU) at (2.5,0.87) {};

                \draw (R) -- (M) node[mid_auto] {4};
                \draw (RR) -- (R) node[mid_auto] {2};
                \draw (RRU) -- (RR) node[mid_auto] {6};
            \end{scope}
            \begin{scope} [xscale = -0.6, yscale = -0.5]
                \node[blk] (R) at (1,0) {};
                \node[wht] (RR) at (2,0) {};
                \tikzset{my642={RR,-60}}
                \node[blk] (RRU) at (2.5,0.87) {};

                \draw (R) -- (M) node[mid_auto] {4};
                \draw (RR) -- (R) node[mid_auto] {2};
                \draw (RRU) -- (RR) node[mid_auto] {6};
            \end{scope}
            
        \end{tikzpicture}
        \hspace{2cm}
        % crab4
        \begin{tikzpicture}
            \node[wht] (M) at (0,0) {};
            \begin{scope} [xscale = 0.6, yscale = 0.5]
                \tikzset{my642={M,90}}
            \end{scope}

            \begin{scope} [xscale = 0.6, yscale = -0.5]
                \node[blk] (R) at (1,0) {};
                \node[wht] (RR) at (2,0) {};
                \tikzset{my642={RR,-60}}
                \node[blk] (RRU) at (2.5,0.87) {};

                \draw (R) -- (M) node[mid_auto] {4};
                \draw (RR) -- (R) node[mid_auto] {2};
                \draw (RRU) -- (RR) node[mid_auto] {6};
            \end{scope}
            \begin{scope} [xscale = -0.6, yscale = -0.5]
                \node[blk] (R) at (1,0) {};
                \node[wht] (RR) at (2,0) {};
                \tikzset{my642={RR,-60}}
                \node[blk] (RRU) at (2.5,0.87) {};

                \draw (R) -- (M) node[mid_auto] {4};
                \draw (RR) -- (R) node[mid_auto] {2};
                \draw (RRU) -- (RR) node[mid_auto] {6};
            \end{scope}
            
        \end{tikzpicture}
        \caption{4 crab-shaped trees}
    \end{subfigure}
    \caption{11 trees of $p=10$ , $q=6$}
    \label{fig:tree_p10q6}
\end{figure}

\begin{example}
    We set $p = 10$ and $q = 6$, i.e. $\Xi = (6^{10},10^6)$, then $g_0 = 2, g_1 = 3$ and $g_2 = 5$. By Proposition \ref{prop:G(1)_when_g_ne_1}, we get :
    \begin{equation*}
        \begin{gathered}
            G(1) = (-1) \frac{1}{2!}[\frac{(\frac{1\times(10+6)}{2}-1)!}{(\frac{1\times10}{g_0})!(\frac{1\times6}{g_0})!}]^{2} + (10+6-1)^{-1} \frac{1}{n_j!} \frac{(\frac{2\times(10+6)}{2}-1)!}{(\frac{2\times10}{g_0})!(\frac{2\times6}{g_0})!} = \\
            - \frac{1}{2!} (\frac{7!}{5!3!})^{2} + \frac{1}{15} \frac{15!}{10!6!} = \frac{133}{15} .
        \end{gathered}
    \end{equation*}
    By Proposition \ref{prop:G(d)_when_g_ne_1}, we get :
    \begin{equation*}
        G(3) = (10+6-1)^{-1} \binom{\frac{2}{3}\frac{10+6}{2}-\frac{1}{3}}{\frac{2}{3}\frac{6}{2}} = \frac{1}{15} \binom{5}{2} = \frac{2}{3} ,
    \end{equation*}
    \begin{equation*}
        G(5) = (10+6-1)^{-1} \binom{\frac{2}{5}\frac{10+6}{2}-\frac{1}{5}}{\frac{2}{5}\frac{10}{2}} = \frac{1}{15} \binom{3}{2} = \frac{1}{5} ,
    \end{equation*}
    Then the total number of $Tree(\Xi)$ is 
    \begin{equation*}
        \#\Tree(\Xi) = G(1) + \varphi(3)G(3) + \varphi(5)G(5) = \frac{133}{15} + 2\times\frac{2}{3} + 4\times\frac{1}{5} = 11 ,
    \end{equation*}
    where the number of $Tree(\Xi, 1)$, $Tree(\Xi, 3)$ and $Tree(\Xi, 5)$ is
    \begin{equation*}
        \#Tree(\Xi, 1) = \mu(\frac{1}{1}) G(1) + \mu(\frac{3}{1}) G(3) + \mu(\frac{5}{1}) G(5) = \frac{133}{15} - \frac{2}{3} - \frac{1}{5} = 8 ,
    \end{equation*}
    \begin{equation*}
        \#Tree(\Xi, 3) = 3 \mu(\frac{3}{3}) G(3) = 2 ,
    \end{equation*}
    \begin{equation*}
        \#Tree(\Xi, 5) = 5 \mu(\frac{5}{5}) G(5) = 1 .
    \end{equation*}
    All the trees are shown in Figrue \ref{fig:tree_p10q6}.
\end{example}

\bigskip
\bibliographystyle{plain}
\bibliography{HCMU_v2}

\begin{thebibliography}{10}

\bibitem{Calabi82extrm}
Eugenio Calabi.
\newblock {\em Extremal K\"{a}hler Metrics}, volume 102 of {\em Annals of
  Mathematics Studies}, pages 259--290.
\newblock Princeton University Press, Princeton, 1982.

\bibitem{Calabi85extrm}
Eugenio Calabi.
\newblock Extremal {K}\"ahler metrics. {II}.
\newblock In {\em Differential geometry and complex analysis}, pages 95--114.
  Springer, Berlin, 1985.

\bibitem{CCW05}
Qing Chen, Xiuxiong Chen, and Yingyi Wu.
\newblock The structure of {HCMU} metric in a {$K$}-surface.
\newblock {\em Int. Math. Res. Not.}, (16):941--958, 2005.

\bibitem{chen2009existence}
Qing Chen and Wu~Yingyi.
\newblock Existence and explicit constructions of hcmu metrics on s2 and t2.
\newblock {\em Pacific journal of mathematics}, 240(2):267--288, 2009.

\bibitem{Cxx98}
Xiuxiong Chen.
\newblock Weak limits of {R}iemannian metrics in surfaces with integral
  curvature bound.
\newblock {\em Calc. Var. Partial Differential Equations}, 6(3):189--226, 1998.

\bibitem{Cxx99}
Xiuxiong Chen.
\newblock Extremal {H}ermitian metrics on {R}iemann surfaces.
\newblock {\em Calc. Var. Partial Differential Equations}, 8(3):191--232, 1999.

\bibitem{chen2000obstruction}
Xiuxiong Chen.
\newblock Obstruction to the existence of metric whose curvatures has umbilical
  hessian in a k-surface.
\newblock {\em Communications in analysis and geometry}, 8(2):267--300, 2000.

\bibitem{kochetkov2015enumeration}
Yu~Yu Kochetkov.
\newblock Enumeration of one class of plane weighted trees.
\newblock {\em Journal of Mathematical Sciences}, 209(2):282--291, 2015.

\bibitem{lin2002explicit}
Chang~Shou Lin and Xiaohua Zhu.
\newblock Explicit construction of extremal hermitian metrics with finite
  conical singularities on s2.
\newblock {\em Communications in Analysis and Geometry}, 10(1):177--216, 2002.

\bibitem{LinZhu02}
Changshou Lin and Xiaohua Zhu.
\newblock Explicit construction of extremal hermitian metrics with finite
  conical singularities on $s^2$.
\newblock {\em Communications in Analysis and Geometry}, 10:177--216, 2002.

\bibitem{lu2025modulispacehcmusurfaces}
Sicheng Lu and Bin Xu.
\newblock The moduli space of hcmu surfaces, 2025.

\bibitem{MyjWzq24}
Yingjie {Meng} and Zhiqiang {Wei}.
\newblock {Classification of non-CSC extremal K{\"a}hler metrics on K-surfaces
  $S^2_{\{\alpha\}}$ and $S^2_{\{\alpha,\beta\}}$}.
\newblock {\em Math. Ann.}, pages 1--28, August 2024.

\bibitem{WangZhu00}
Guofang Wang and Xiaohua Zhu.
\newblock Extremal {H}ermitian metrics on {R}iemann surfaces with
  singularities.
\newblock {\em Duke Math. J.}, 104(2):181--210, 2000.

\end{thebibliography}

\end{document}